\documentclass[a4paper,11pt]{article}

\usepackage{amssymb}
\usepackage{amsfonts}
\usepackage{theorem}
\usepackage{bbm}
\usepackage{graphicx}
\usepackage{amsmath}
\usepackage{enumerate}
\usepackage{indentfirst}
\usepackage{color}
\usepackage{mathtools}
\usepackage{hyperref,enumitem}
\usepackage{enumitem}

\newtheorem{thm}{Theorem}
\newtheorem{defin}{Definition}
\newtheorem{remark}{Remark}
\newtheorem{prop}{Proposition}
\newtheorem{lemma}{Lemma}
\newtheorem{lemmaA}{Lemma}
\newtheorem{coro}{Corollary}

\newenvironment{proof}[1][Proof]{\textbf{#1.} }{\hfill $\square$}

\newcommand{\Pro}{\mathbf{P}}

\newcommand{\E}{\mathbb{E}}
\newcommand{\R}{\mathbbm{R}}

\newcommand{\fG}{\mathfrak{G}}

\newcommand{\al}{\alpha}
\newcommand{\eps}{\varepsilon}
\newcommand{\ds}{\displaystyle}

\newcommand{\bD}{\mathbb{D}}

\newcommand{\bL}{\mathbb{L}}

\newcommand{\rma}{\mathrm{a}}

\newcommand{\mP}{\mathbf{P}}

\newcommand{\mE}{\mathbf{E}}

\newcommand{\cF}{\mathcal{F}}

\newcommand{\cL}{\mathcal{L}}

\newcommand{\cK}{\mathcal{K}}

\title{On the fundamental solution of heat and  stochastic heat equations}
\author{Marina Kleptsyna\thanks{Laboratoire Manceau de Math\'ematiques, Le Mans Universit\'e, Avenue O. Messiaen, 72085 Le Mans, Cedex 9, France.
e-mail: {\tt  marina.kleptsyna@univ-lemans.fr, alexandre.popier@univ-lemans.fr,}}
\,\,\, Andrey Piatnitski \thanks{The Arctic University of Norway, campus Narvik,  P.O.Box 385,
8505 Narvik, Norway\ and \ Institute for Information Transmission Problems of RAS, 19, Bolshoy Karetny per., Moscow 127051,
Russia.\ e-mail: {\tt apiatnitski@gmail.com}}
\,\,\, and Alexandre Popier\footnotemark[3]
}

\date{\today}

\begin{document}
\maketitle

\begin{abstract}
We consider the generic divergence form second order parabolic equation with coefficients that are regular in the spatial variables and just measurable in time. We show that the spatial derivatives of its fundamental solution admit upper bounds that agree with the Aronson type estimate and only depend on the ellipticity constants of the equation and the $L^\infty$ norm of the spatial derivatives of its coefficients.

 We also study the corresponding stochastic partial differential equations and prove that under natural assumptions
  on the noise the equation admits a mild solution, given by anticipating stochastic integration.
\end{abstract}

\vspace{0.5cm}
35C15, 35K08.


\section{Introduction}\label{s_intro}

In the first part of the paper we study the fundamental solution $\Gamma=\Gamma(x,t,y,s)$
 of the parabolic equation
\begin{equation}\label{eq:PDE_fund_sol}
\frac{\partial u}{\partial t} \displaystyle(x,t)=\mathrm{div} \Big[\rma\Big(x,t\Big)\nabla u(x,t) \Big],
\quad (x,t)\in\mathbb R^d\times(0,T].
\end{equation}

The existence  of the fundamental solution $\Gamma$ and the description of its properties is an old story that has given rise to a vast literature (see among others \cite{frie:64,lady:solo:ural:67,porp:eide:84,eide:zhit:98,kryl:08} and the references therein). One of the most famous result in this field is the Aronson estimate (see Inequality \eqref{eq:aronson_estim} and \cite[Theorem 7]{aron:68}), which holds under the uniform ellipticity condition on the diffusion matrix $\rma$ (see condition \ref{H1}). No regularity assumption on the coefficients of $\rma$ is required. To obtain similar estimates on the spatial  derivatives of $\Gamma$, it is usually assumed in the existing literature that the matrix $\rma$ is H\"older continuous w.r.t. both $x$ and $t$ (see \cite{lady:solo:ural:67}, Chapter IV, sections 11 to 13 or \cite{frie:64}, Chapter I): for some $\hbar \in (0,1)$
$$
\left| \rma \left(x,t\right) - \rma \left(x',t' \right) \right| \leq K_\rma \left( |x-x'|^\hbar + \left| t-t'\right|^{\hbar/2}\right).
$$
Notice that this setting is not well adapted to the stochastic framework, for example if $\rma(x,t) = a(x,\xi_t)$ where $\xi$ is a diffusion process. Indeed, in this case  the constant $K_\rma$ depends on the continuity properties of $\xi$ and is random
(see for example \cite{baud:14} for details). Hence the constants in the estimate of $\nabla_x \Gamma$
need not be uniformly bounded  if we follow directly this construction.

Our first goal in the paper is to obtain  Aronson type estimates for the spatial derivatives of $\Gamma$, without any regularity assumption on the dependence $t \mapsto \rma(x,t)$. We impose only  a uniform Lipschitz continuity condition on the dependence $x\mapsto \rma(x,t)$. Then the upper bounds only depend on the ellipticity constants and $L^\infty$ norm of the gradient of the coefficients
(see Theorem \ref{thm:general_result_Aronson_estim}).

There is a vast literature devoted to the behaviour of fundamental solutions of parabolic equations.
One of the questions of interest is how does the behaviour of fundamental solution of a parabolic equation
defined on a (non-compact) Riemannien manifold  depend on the properties of the metric.
This question was studied in the works \cite{DAVIES198816}, \cite{Grigoryan1994}, \cite{saloff-coste2010}, \cite{grig:09} and some others.
Similar problems were considered  for operators defined on fractal sets, see \cite{barlow2003kernels},  on groups, see \cite{MR1218884},  and on metric spaces,
see \cite{grig:12}, \cite{MR3809455}. It is usually assumed that the Radon measure on the metric space satisfies the so-called volume doubling property.\\
There is also a number of papers that focus on interior bounds for heat kernel, see for example \cite{LIERL20144189,lierl:18}. \\
Fundamental solutions of parabolic equations with time dependent coefficients have been investigated in \cite{Guenther2002}, \cite{Daners_2000}.\\
A number of estimates for the derivatives of heat kernels on manifolds and metric spaces was obtained in
\cite{JIAYU1991293}, \cite{GRIGORYAN1995363}, \cite{Stroock1998UpperBO}, \cite{hu:li:10} and some other works including recent work \cite{coul:jiang:kosk:19}.\\
{Several papers also focuses on the fundamental solution of diffusion operators generated by Markov processes with jumps \cite{chen:hu:xie:17,chen:zhang:16} or Dirichlet forms in \cite{chen:kuma:wang:19}. }\\
However, in all the above mentioned works the question studied in the present paper has not been raised.\\

When our paper was submitted we leant that a number of results closely related to that of  Theorem \ref{thm:general_result_Aronson_estim}
have been obtained in the recent work \cite{chen_at_al:2017}. In this work, for parabolic operators in non-divergence form with time dependent coefficient, the regularity of heat kernel and solutions w.r.t. spatial variables is studied.
%
In particular, the result of our
Theorem \ref{thm:general_result_Aronson_estim} can be derived from the results of this work.  However,
the approach used in \cite{chen_at_al:2017} is rather different.


\bigskip
In the second part of this paper we deal with the following stochastic heat equation:
\begin{equation}\label{eq:SPDE}
d v\displaystyle(x,t)-\mathrm{div} \Big[\rma \Big(x,t\Big)\nabla v(x,t) \Big]\,dt = \displaystyle G\left(x,t\right)\,dB_t\
\end{equation}
with the initial condition $v(x,0)=0$ (see Remark \ref{rem:initial_cond} for more general initial value). $B$ is a standard Brownian motion, generating the filtration $\mathbb F = (\cF_t, \ t \geq 0)$. The matrix $\rma$ is supposed to be a measurable function from $\R^d \times [0,+\infty[ \times \Omega$ into $\R^{d\times d}$ and for each $(x,t) \in \R^d \times [0,+\infty[$, $\rma (x,t)$ is $\cF_t$-measurable. This stochastic partial differential equation (SPDE in short) in divergence form is somehow classical and among many other we refer to the books \cite{frie:64, lady:solo:ural:67} on PDE in divergence form, \cite{dapr:zabc:14,dala:khos:09,kryl:rozo:79,pard:93,wals:86} and the references therein on SPDE.  {The results of these works have than been extended in several directions, among them are: H\"ormander's condition \cite{kryl:99,kryl:14}, H\"older spaces \cite{chow:jiang:94,miku:prag:14}, $L^p$-spaces \cite{deni:mato:stoi:05,kryl:96,miku:prag:13,miku:rozo:01}, Laplace-Beltrami operator \cite{sowe:98}. }

Our aim is to prove that the SPDE  in \eqref{eq:SPDE} admits a mild solution $v$ given by:
\begin{eqnarray} \label{eq:mild_sol}
v(x,t) & = & \int_0^t \int_{\R^d} \Gamma(x,t,y,s) G\left(y,s \right)dy dB_{s},
\end{eqnarray}
where $\Gamma$ is the fundamental solution of the equation in
\eqref{eq:PDE_fund_sol}.

If the matrix $\rma$ is deterministic, $\Gamma$ is also deterministic and the existence of a mild solution $v$ given by \eqref{eq:mild_sol} is well known (see \cite[Chapter 5]{wals:86}). However, when $\rma$ is random, the stochastic integral in \eqref{eq:mild_sol} has to be defined properly since $\Gamma(x,t,y,s)$ is measurable w.r.t. the $\sigma$-field $\mathcal{F}_{t}$ generated by the random variables $B_u$ with $u \leq t$. In other words Equation \eqref{eq:mild_sol} involves an anticipating integral. To our best knowledge, there is only one work on this topic by Alos et al. \cite{alos:leon:nual:99}. Compared to our setting, the authors in \cite{alos:leon:nual:99} consider a space-time Wiener process, but the matrix $\rma$ is H\"older continuous in time\footnote{At the end of \cite[section 5]{alos:leon:nual:99}, the authors make a remark and give an example on this time regularity assumption.} (condition (A3) in \cite{alos:leon:nual:99}).

From the first part of this paper, we know that $\Gamma$ and its spatial derivative admit Aronson's type upper bounds and we extend these bounds to the Malliavin derivatives of $\Gamma$, again without regularity assumption on $\rma$ w.r.t. $t$ (see Theorem \ref{thm:malliavin_deriv_fund_sol} and, in the diffusion case, Corollary \ref{coro:aronson_estim_diff_case}).

Finally, since our noise is a one parameter Brownian motion, we also want to obtain a regular mild solution $v$ on $\R^d \times (0,T)$ in the sense of Definition \ref{def:weak_sol} of Equation \eqref{eq:SPDE}. Compared to \cite{alos:leon:nual:99}, since we have no space noise, we do not impose any condition on the dimension $d$ and our solution is derivable w.r.t. $x$ (see Theorem \ref{thm:reg_mild_sol} and Corollary \ref{coro:reg_mild_sol}).

In a recent paper paper \cite{pasc:pesc:18} a similar subject is handled with a parametrix construction. However, since the studied operator is not in the divergence form, the authors have to impose more regularity assumptions on the diffusion matrix $\rma$. Also, the SPDEs investigated in this paper are rearranged in such a way that the anticipating stochastic calculus can be avoided.

\bigskip
The paper is organized as follows. In Section \ref{sect:deriv_fund_sol} we  consider the generic heat equation \eqref{eq:PDE_fund_sol} and its fundamental solution $\Gamma$. We prove that the spatial derivatives of $\Gamma$ admit an Aronson's type upper bound, without any time regularity condition on $\rma$. Our result is presented in Theorem \ref{thm:general_result_Aronson_estim}.

In the next section \ref{sect:malliavin_deriv_fund_sol}, we assume that the matrix $\rma$ is random. Using the arguments developed in the previous section, we a number of estimates for the Malliavin derivative of $\Gamma$ and of
$\nabla \Gamma$, see Theorem \ref{thm:malliavin_deriv_fund_sol}. We also study the particular case where the randomness
is given by the solution of a SDE (diffusion case, section \ref{ssect:diff_case}).

In Section \ref{sect:mild_sol} we construct a mild solution $v$ of the SPDE in \eqref{eq:SPDE}. Here we use anticipating calculus and the properties of the fundamental solution $\Gamma$ of a parabolic equation with random coefficients. In the first part of this section we provide our assumptions and formulate the main result concerning a mild solution (Theorem \ref{thm:reg_mild_sol} and Corollary \ref{coro:reg_mild_sol} in the diffusion case). Section \ref{ssect:mild_sol_cons} is devoted to the proof of those results.

\section{Estimate for the spatial derivative of the fundamental solution} \label{sect:deriv_fund_sol}

Our goal here is to obtain an upper bound for the derivative of the fundamental solution $\Gamma$ for the PDE \eqref{eq:PDE_fund_sol}. On the matrix $\rma : \R^d \times [0,+\infty) \to \R^{d\times d}$ we impose the following conditions.
\begin{enumerate}[label=\textbf{(H\arabic*)}]
\item \label{H1} Uniform ellipticity. For any $(t,x,\zeta) \in \R_+\times \R^d \times \R^d$
$$
\lambda^{-1} |\zeta|^2\leq \rma(x,t)\zeta\cdot\zeta\leq \lambda |\zeta|^2.
$$
\item \label{H2} The matrix $\rma$ is measurable on $\R^d \times \R_+$, and for any $t \geq 0$ the function $\rma(\cdot,t)$ is of class $C^1$ w.r.t. $x \in \R^d$. Moreover,  there is a constant $K_\rma$ such that for all $t$ and $x$
$$
|\nabla \rma(x,t) |  \leq K_\rma.
$$
\end{enumerate}
We denote by $\cL$ the operator: $\cL = \mathrm{div} \Big[\rma \Big(x,t\Big)\nabla \Big]$,
then \eqref{eq:PDE_fund_sol} can be written:
\begin{equation*}
\frac{\partial u}{\partial t} \displaystyle(x,t)=\cL u (x,t).
\end{equation*}
It is well known (see among other \cite{aron:68} or \cite{eide:zhit:98}) that under condition \ref{H1} there exist two constants $\varsigma>0$ and $\varpi>0$ depending only on the constant $\lambda$ in Assumption \ref{H1} and the dimension $d$, such that
\begin{equation}\label{eq:aronson_estim}
0\leq \Gamma(x,t,y,s)  \leq  g_{\varsigma,\varpi}(x-y,t-s);
\end{equation}
here and in what follows, for two positive constants $c$ and $C$, the function $g_{c,C}(x,t)$ is defined by
$$
\textstyle
g_{c,C}(x,t) = c\, t^{-\frac{d}{2}}  \exp \left(-\frac{C|x|^2}{t}  \right),\qquad t> 0,\ \ x \in \R^d.
$$
Inequality \eqref{eq:aronson_estim} is called the {\bf Aronson estimate}\footnote{The function $\Gamma$ has a lower bound similar to the upper bound (see \cite[Theorem 7]{aron:68})}.  Our first result reads.
\begin{thm} \label{thm:general_result_Aronson_estim}
If the matrix $\rma=\rma(x,t)$ \ satisfies the uniform ellipticity condition {\rm \ref{H1}} and the regularity condition {\rm \ref{H2}}, then the (weak) fundamental solution $\Gamma$ of equation \eqref{eq:PDE_fund_sol}
admits the following estimate: there exist two constants $\varrho>0$ and $\varpi>0$ such that
\begin{equation}\label{eq:aronson_estim_der}
| \nabla_x \Gamma(x,t,y,s) | \leq  \frac{1}{\ds \sqrt{t-s}} \ g_{\varrho,\varpi}(x-y,t-s);
\end{equation}
here $\varpi$ depends only on the uniform ellipticity constant $\lambda$ and the dimension $d$, while $\varrho$ might
also depend on $K_\rma$ and on $T$.
\end{thm}
Weak fundamental solution is defined in \cite[Definition VI.6]{eide:zhit:98}. Let us emphasize that these estimates are coherent with \cite[Theorem VI.4]{eide:zhit:98}. The novelty is that the regularity of $\rma$ w.r.t. $t$ is not required.
The rest of this section is devoted to the proof of this theorem.

\subsection{When $\rma$ does not depend on $x$.}

First assume that $\rma$ just depends on $t$. In this case the fundamental solution $\Gamma$ is denoted by $Z$ and is given by the formula: for any $s<t$ and $(x,y)\in (\R^d)^2$
\begin{equation}\label{eq:fourier_transf_fund_sol}
Z(x-y,t,s) = \frac{1}{(2\pi)^{d/2}} \int_{\R^d} e^{i\zeta (x-y)} V(t,s,\zeta) d\zeta,
\end{equation}
where $V$ is the following function:
$$V(t,s,\zeta) = \exp \left( - \left\langle \int_s^t \rma(u) du\  \zeta,\zeta \right\rangle \right).$$
Due to Condition \ref{H1} the matrix $\rma$ verifies the estimates
$$\lambda^{-1} (t-s) |\zeta|^2 \leq \left\langle \int_s^t \rma(u) du\  \zeta,\zeta \right\rangle  \leq \lambda (t-s) |\zeta|^2.$$
From the above expression for $Z$, we deduce that for any $k\geq 1$ and $1 \leq j\big._\ell \leq d$ with $1\leq \ell \leq k$
\begin{eqnarray*}
\partial^k_{x_{j_1} \ldots x_{j_k} } Z(x-y,t,s) & = &  \frac{(i)^{k} }{(2\pi)^{d/2}} \int_{\R^d} e^{i\zeta (x-y)} V(t,s,\zeta) (\zeta_{j_1} \ldots  \zeta_{j_k}) d\zeta.
\end{eqnarray*}
As in \cite{frie:64}, Chapter 9, Theorem 1, we obtain that:
\begin{equation} \label{eq:aronson_deriv_parametrix}
|\partial^k_{x_{j_1} \ldots x_{j_k} }  Z(x-y,t,s)| \leq \frac{1}{(t-s)^{k/2}} \ g_{\varsigma,\varpi}(x-y,t-s).
\end{equation}
In particular the Aronson estimates \eqref{eq:aronson_estim} and \eqref{eq:aronson_estim_der} can be derived.

Now we define the parametrix, also denoted by $Z$, as the fundamental solution of \eqref{eq:PDE_fund_sol} for $\rma(\mathfrak z,t)$ where $\mathfrak z \in \R^d$ is a fixed parameter:
\begin{equation*}
\frac{\partial u}{\partial t} \displaystyle(x,t)=\mathrm{div} \Big[\rma(\mathfrak z,t)\nabla u(x,t)\Big].
\end{equation*}
We have again the representation
\begin{equation} \label{eq:repre_parametrix}
\forall s \leq t, \quad Z( x-y,t,s,\mathfrak z) = \frac{1}{(2\pi)^{d/2}} \int_{\R^d} e^{i\zeta (x-y)} V(t,s,\zeta,\mathfrak z) d\zeta,
\end{equation}
with
$$V(t,s,\zeta,\mathfrak z) = \exp \left( - \left\langle \int_s^t \rma(\mathfrak z,u) du\  \zeta,\zeta \right\rangle \right).$$
The above arguments give Estimates \eqref{eq:aronson_estim}  and \eqref{eq:aronson_estim_der}.
The following statement is equivalent to Lemma 5 in \cite{frie:64}, Chapter 9, Section 3 (see also \cite[Theorem I.3.2]{frie:64}).

In the next section, we use the parametrix method to construct $\Gamma$ when $\rma$  depends on both $x$ and $t$. The following technical result is used several times.
\begin{lemma} \label{lem:improper_integrals}
Suppose that $f$ is a measurable function on $\R^d \times [0,+\infty)$ that satisfies the estimate
$$
|f(x,t)|\leq k \exp(\mathfrak k |x|^2)
$$
for some constants $k$ and $\mathfrak k < \varpi /T$. Then the integral
$$F(x,t) = \int_0^t \left( \int_{\R^d} Z( x-\zeta,t,s,\zeta) f(\zeta,s) d\zeta \right) ds $$
is well defined for $0 \leq t \leq T$, continuous on $\R^d \times [0,T]$, and the derivative $\nabla_x F$ exists for $0 < t \leq T$ and
$$
\nabla_x F(x,t) = \int_0^t \left( \int_{\R^d} \nabla_x Z( x-\zeta,t,s,\zeta) f(\zeta,s) d\zeta \right) ds.
$$
\end{lemma}
\begin{proof}
We skip the proof of this Lemma because  it is the same as the proof of Lemma IX.5 in \cite{frie:64} (see also \cite{frie:64}, Chapter 1, Section 3 for more details).
\end{proof}

\subsection{Parametrix method and the estimate on the gradient}

The parametrix method suggests to construct $\Gamma$ in the form
\begin{eqnarray} \nonumber
\Gamma(x,t,y,s) &= &Z(x-y,t,s,y) \\ \label{eq:def_fund_sol_param_method}
& +& \int_s^t \int_{\R^d} Z(x-\zeta,t,r,\zeta) \Phi(\zeta,r,y,s) d\zeta dr.
\end{eqnarray}
If the function $\Phi$ is measurable and satisfies a suitable growth condition, we can apply Lemma \ref{lem:improper_integrals}. Then $\Gamma$ is the fundamental solution if and only if
\begin{equation*}
 \Phi(x,t,y,s) = \cK(x,t,y,s) + \int_s^t \int_{\R^d} \cK(x,t,\zeta,r) \Phi(\zeta,r,y,s) d\zeta dr,
\end{equation*}
where
\begin{eqnarray*}
\cK(x,t,y,s) & = &\mathrm{div} \left[ \left( \rma\big(x,t\big)- \rma \big(y,t\big) \right) \nabla_x Z(x-y,t,s,y)\right] .
\end{eqnarray*}
Notice that in the expression $\rma\big(x,t\big) - \rma\big(y,t\big)$, the matrix is evaluated two times at the same time $t$.
Hence formally the function $\Phi$ is the sum of iterated kernels
\begin{equation} \label{eq:def_Phi_2}
\Phi(x,t,y,s)  = \sum_{m=1}^\infty \cK_m (x,t,y,s)
\end{equation}
with $\cK_m$ defined by
$$\cK_m (x,t,y,s) = \int_s^t \int_{\R^d} \cK(x,t,\zeta,r) \cK_{m-1}(\zeta,r,y,s) d\zeta dr.$$
Let us follow the scheme of \cite{frie:64} to obtain \eqref{eq:aronson_estim_der}.
Remark that continuity of $\rma$ w.r.t. $t$ is not assumed. We will use the following notations: $\rma_i$ is the $i$-th column of $\rma$, $\gamma$ is the vector-function such that
$$
\gamma_i(x,t) = \mathrm{div} (\rma_i(x,t)) = \sum_{j=1}^n  \frac{\partial \rma_{ji}}{\partial x_j} (x,t).
$$
Note that under \ref{H2}, $\gamma$ is bounded. The kernel $\cK$ satisfies:
\begin{eqnarray} \nonumber
\cK(x,t,y,s) & = & \sum_{i,j=1}^n (\rma_{ij}(x,t)-\rma_{ij}(y,t)) \frac{\partial^2 Z}{\partial x_i \partial x_j}(x-y,t,s,y) \\ \label{eq:kernel_first_step}
& + & \sum_{i=1}^n  \gamma_i(x,t) \frac{\partial Z}{\partial x_i } (x-y,t,s,y).
\end{eqnarray}
\begin{lemma} \label{lem:estim_phi_eps}
Under {\rm \ref{H1}} and {\rm \ref{H2}}, the series in \eqref{eq:def_Phi_2} converge. The sum $\Phi$ is measurable and satisfies the estimate
\begin{equation} \label{eq:estim_Phi_eps}
|\Phi(x,t,y,s)| \leq  \frac{1}{\sqrt{t-s}} g_{\varrho,\varpi}(x-y,t-s).
\end{equation}
The constants $\varrho$ and $\varpi$ depend on $\lambda$ and $d$, whereas $\varrho$ also depends on the Lispchitz constant $K_\rma$ 
and on $T$.
\end{lemma}
\begin{proof}
From estimate \eqref{eq:aronson_deriv_parametrix} considering  Lipschitz continuity of $\rma$, we obtain
\begin{eqnarray*}
\left| \cK(x,t,y,s) \right|& \leq & K_\rma    |x-y| \frac{1}{t-s} g_{\varsigma,\varpi}(x-y,t-s) \\[-1mm]
& +& K_\rma    \frac{1}{\sqrt{t-s}} g_{\varsigma,\varpi}(x-y,t-s) \\[-1mm]
& \leq & \frac{1}{\sqrt{t-s}} g_{\varrho,\varpi}(x-y,t-s).
\end{eqnarray*}
Again $\varsigma$, $\varpi$ or $\varrho$ may differ from line to line. Thus $\cK$ satisfies  inequality (4.6) of \cite{frie:64}, Chapter 9, Section 4.
Then the convergence of the series in \eqref{eq:def_Phi_2} can be proved by the same arguments. Indeed,
by Lemma IX.7 in \cite{frie:64} for any $\eta$, $0< \eta < 1$, there is a constant $M(\eta,\varpi)>0$ depending on $\eta$, $\varpi$ and  $d$ such that
\begin{eqnarray*}
&& \left| \cK_2(x,t,y,s) \right|  \leq  \int_s^t \int_{\R^d} \left|  \cK(x,t,\zeta,r) \right| \left| \cK(\zeta,r,y,s) \right| d\zeta dr \\
 && \quad \leq  \int_s^t  \int_{\R^d}  \frac{1}{\sqrt{(t-r)(r-s)} }\ g_{\varrho,\varpi}(x-\zeta,t-r)g_{\varrho,\varpi}(\zeta-y,r-s) d\zeta dr \\
 && \quad \leq \int_s^t   \frac{M(\eta,\varpi)\varrho^2 }{\sqrt{(t-r)(r-s)} }\frac{1}{\ds (t-s)^{\frac{d}{2}}}  \exp \left(-\varpi(1-\eta) \frac{|x-y|^2}{t-s}  \right)dr
\end{eqnarray*}
By direct computation (see also Lemma I.2 in \cite{frie:64})
$$
\int_s^t {\textstyle \frac{1}{\sqrt{(t-r)(r-s)} } dr} = \pi.
$$
Thereby there exist two constants $\varrho>0$ and $\varpi>0$ such that
$$
\left| \cK_2(x,t,y,s) \right| \leq  g_{\varrho,\varpi}(x-y,t-s).
$$
Iterating this computation we obtain by induction for $m\geq 2$:
$$\left| \cK_m(x,t,y,s) \right| \leq \frac{M^m}{(1+m/2)!} (t-s)^{m/2-1} g_{\varrho,\varpi}(x-y,t-s)$$
where $M$ is a constant depending on $\varrho$ and $\varpi$, and the symbol $(\cdot)!$ stands for the gamma function (see the proof of Theorem IX.2 in \cite{frie:64} for the details). The convergence of the series and  estimate \eqref{eq:estim_Phi_eps} can be then deduced. Namely,
\begin{eqnarray*}
|\Phi(x,t,y,s)| &\leq & \frac{1}{\sqrt{t-s}} g_{\varrho,\varpi}(x-y,t-s) \Big[  \sum_{m\geq 1}  \frac{M^m}{(1+m/2)!} (t-s)^{(m-1)/2} \Big]\\
& = &  \frac{1}{\sqrt{t-s}} g_{\varrho,\varpi}(x-y,t-s) \Theta(t-s).
\end{eqnarray*}
For $t$ and $s$ in $[0,T]$, we get: $\Theta(t-s) \leq \Theta(T)$.
\end{proof}

Using Lemma \ref{lem:improper_integrals}, we deduce that $\Gamma$ is well-defined, and inequality \eqref{eq:aronson_estim_der} follows from the formula
\begin{eqnarray} \nonumber
\partial_{x_j} \Gamma(x,t,y,s) 	& = & \partial_{x_j}  Z(x-y,t,s,y) \\ \label{eq:def_fund_sol_param_method_deriv}
&+ &\int_s^t \int_{\R^d} \partial_{x_j}  Z(x-\zeta,t,r,\zeta) \Phi(\zeta,r,y,s) d\zeta dr,
\end{eqnarray}
together with estimate \eqref{eq:aronson_deriv_parametrix} on $Z$ and \eqref{eq:estim_Phi_eps} on $\Phi$. We underline that only the properties \ref{H1} and \ref{H2} of $\rma$ are required to obtain \eqref{eq:aronson_estim_der}. This completes the proof of Theorem \ref{thm:general_result_Aronson_estim}.

\section{Malliavin derivative of the fundamental solution}\label{sect:malliavin_deriv_fund_sol}

From now on we suppose that $\rma=\rma(x,t)$ are random fields defined on a probability space $(\Omega,\cF,\Pro)$ that carries a $d$-dimensional Brownian motion $B$ and that the filtration $\mathbb F = (\cF_t, \ t \geq 0)$ is generated by $B$, augmented with the $\Pro$-null sets. The matrix $\rma : \R^d \times [0,+\infty) \times \Omega \to \R^{d\times d}$ depends\footnote{Note that here and in the sequel we follow the usual convention and omit the function argument $\omega$.} also on $\omega$ and we assume that conditions \ref{H1} and \ref{H2} are fulfilled uniformly w.r.t. $\omega$. In particular the ellipticity constant $\lambda$ and the bound $K_\rma$ do not depend on $\omega$. Since \ref{H1} and \ref{H2} hold, by Theorem \ref{thm:general_result_Aronson_estim} the fundamental solution $\Gamma$ of \eqref{eq:PDE_fund_sol} and its spatial derivatives satisfy  estimates \eqref{eq:aronson_estim} and \eqref{eq:aronson_estim_der}.

In order to define properly the stochastic integral in \eqref{eq:mild_sol}, we will use the approach developed in \cite{nual:pard:88} for anticipating integrals and thus Malliavin's derivatives. In what follows we borrow some notations from Nualart \cite{nual:06}.
Recall that $B$ is a $d$-dimensional Brownian motion. Let $f$ be an element of $C^\infty_{\tt p}(\R^{dn})$ (the set of all infinitely many times continuously differentiable functions such that these functions and all  their partial derivatives have at most polynomial growth at infinity) with
$$
f(x)= f(x^{1}_1,\ldots,x^{d}_1;\ldots;x^{1}_n,\ldots,x^{d}_n).
$$
We define a smooth random variable $F$ by:
$$
F = f(B(t_1),\ldots,B(t_n))
$$
for $0\leq t_1 < t_2 < \ldots < t_n \leq T$. The class of smooth random variables is denoted by $\mathcal{S}$. Then the Malliavin derivative $D_t F$ is given by
$$D^j_t(F) = \sum_{i=1}^d \frac{\partial f}{\partial x^j_i} (B(t_1),\ldots,B(t_n)) \mathbf{1}_{[0,t_i]}(t)$$
(see Definition 1.2.1 in \cite{nual:06}). $D_t(F)$ is the $d$-dimensional vector $D_t(F) =(D^j_t(F),\ j=1,\ldots,d)$. Moreover, this derivative $D_t(F)$ is a random variable with values in the Hilbert space $L^2([0,T];\R^d)$. The space $\bD^{1,p}$, $p\geq 1$, is the closure of the class of smooth random variables with respect to the norm
$$\|F\|_{1,p} = \left[ \mE (|F|^p) + \mE \left( \|DF\|_{L^2([0,T];\R^d)}^p \right) \right]^{1/p}.$$
For $p=2$, $\bD^{1,2}$ is a Hilbert space. Then by induction we can define $\bD^{k,p}$ the space of $k$-times differentiable random variables where the $k$ derivatives are in $L^p(\Omega)$. Finally
$$\bD^{k,\infty} = \bigcap_{p\geq 1} \bD^{k,p}, \quad \bD^{\infty} = \bigcap_{k\in \mathbb{N}} \bD^{k,\infty}.$$

For the Malliavin differentiability property of $\Gamma$, we use the approach developed in Al\`os et al. \cite{alos:leon:nual:99}.
We assume that, in addition to \ref{H1} and \ref{H2},  the matrix $\rma$ possesses the following properties:
\begin{enumerate}[label=\textbf{(H\arabic*)}]
\setcounter{enumi}{2}
\item \label{H3} For each $(x,t) \in \R^d \times [0,+\infty)$, $\rma (x,t)$ is a $\cF_t$-measurable random variable.
\item \label{H4} For each $(x,t) \in \R^d \times [0,+\infty)$ the random variable $\rma (x,t)$ belongs to $\bD^{1,2}$.
\item \label{H5}  There exists a non negative process $\psi$ such that for any $t \in [0,T]$ and any $x \in \R^d$,
$$
|D_r \rma(x,t)| + |D_r \nabla \rma(x,t)| \leq \psi(r).
$$
Moreover, $\psi$ satisfies the integrability condition: for some $p>1$
$$
\mE \left( \int_0^T \psi(r)^{2p} dr \right)  < +\infty.
$$
\end{enumerate}
Note that if \ref{H5} holds, then for all $(x,x',t) \in \R^d \times \R^d \times \R_+$
$$
|D_r \rma(x,t) - D_r \rma(x',t)| \leq \psi(r) |x-x'|.
$$
Indeed
$$\rma_{ij}(x,t) - \rma_{ij}(x',t) = \int_0^1 \nabla \rma_{ij}(x'+ \theta(x-x'),t)  d\theta(x-x').$$
We differentiate both sides in the Malliavin sense and we use the estimate on $D_r \nabla \rma$. Our second main result is
\begin{thm} \label{thm:malliavin_deriv_fund_sol}
Under conditions {\rm \ref{H1}--\ref{H5}}, the fundamental solution $\Gamma$ of \eqref{eq:PDE_fund_sol} and its spatial derivatives belong to $\bD^{1,2}$ for every $(t,s)\in[0,T]^2$, $s< t$ and $(x,y)\in (\R^n)^2$. Moreover, there exist two constants $\varrho$ and $\varpi$ that depend only on the uniform ellipticity constant $\lambda$, the dimension $d$, on $K_\rma$ and on $T$, such that
\begin{equation} \label{eq:estim_mall_der_Gamma_bis}
|D_r \Gamma(x,t,y,s) | \leq \psi(r) g_{\varrho,\varpi}(x-y,t-s),
\end{equation}
and
\begin{equation}\label{eq:estim_mall_space_der_Gamma_bis}
|D_r \nabla_x \Gamma(x,t,y,s) | \leq \frac{\psi(r)}{\sqrt{t-s}} g_{\varrho,\varpi}(x-y,t-s) .
\end{equation}
The quantity $\psi$ is defined by \eqref{eq:def_mall_deriv_xi_eps}. Finally $\Gamma$ and $D_r \Gamma$ are continuous w.r.t. $(x,y) \in \R^{2d}$ and $0\leq s < t \leq T$.
\end{thm}
Let us emphasize that the constant $\varpi$ depends only on the uniform ellipticity constant $\lambda$  and the dimension $d$, whereas the constant $\varrho$ also depends on $K_\rma$ and $T$.
%

\subsection{Proof of Theorem \ref{thm:malliavin_deriv_fund_sol}}

Let us remark that the construction of $\Gamma$ in Section \ref{sect:deriv_fund_sol} applies pathwise, $\omega$ by $\omega$. We want to prove now that in the framework of this section $ \Gamma$ is also Malliavin differentiable. As a straightforward consequence of \ref{H3} one obtains that for any $s<t$, the random variables $Z(x-y,t,s)$, $\Phi(x,t,y,s)$ and $ \Gamma$ are $\cF_{t}$-measurable.

Let us first assume that $\rma$ does not depend on $x$ and consider the Malliavin derivative of $Z$. From the representation \eqref{eq:fourier_transf_fund_sol}, this derivative can be computed explicitly: for $j=1,\ldots,d$
\begin{eqnarray*}
&& D^j_r Z(x-y,t,s) = \frac{1}{(2\pi)^{d/2}} \int_{\R^d} e^{i\zeta (x-y)} D^j_r V(t,s,\zeta) d\zeta \\
&&\qquad = -\frac{1}{(2\pi)^{d/2}} \int_{\R^d} e^{i\zeta (x-y)} V(t,s,\zeta) \left\langle  \int_{s}^{t} D^j_r \rma(u) du \ \zeta,\zeta \right\rangle d\zeta.
\end{eqnarray*}
Thus
$$D^j_r Z(x-y,t,s) = \mbox{Trace} \left[ \left(\int_{s}^{t} D^j_r \rma(u) du \right) \partial^2_x Z(x-y,t,s) \right] .$$
Therefore,
$$|D^j_r Z(x-y,t,s) | \leq \left|\int_{s}^{t} D^j_r \rma(u) du  \right| \frac{1}{t-s} \ g_{\varsigma,\varpi}(x-y,t-s).$$
Since  the Malliavin derivative of $\rma$ is bounded by $\psi(r)$, we obtain:
\begin{equation*}
|D_r Z(x-y,t,s) | \leq \psi(r)  g_{\varsigma,\varpi}(x-y,t-s).
\end{equation*}
This yields \eqref{eq:estim_mall_der_Gamma_bis}. Similar computations give:
$$
D_r \partial_{x_j} Z(x-y,t,s) = -i \frac{1}{(2\pi)^{d/2}} \int_{\R^d} e^{i\zeta (x-y)} V(t,s,\zeta)
 \left\langle  \int_{s}^{t} D^j_r \rma(u) du \ \zeta,\zeta \right\rangle \zeta_j d\zeta.
$$
Using the estimate on the third derivative of $Z$ w.r.t. $x$, we obtain \eqref{eq:estim_mall_space_der_Gamma_bis}:
\begin{equation} \label{eq:estim_mall_space_der_Z}
|D_r \partial_{x_i} Z(x-y,t,s) | \leq \psi(r)\frac{1}{(t-s)^{1/2}} \ g_{\varsigma,\varpi}(x-y,t-s) .
\end{equation}
In other words if $\rma$ does not depend on $x$, estimates \eqref{eq:aronson_estim}, \eqref{eq:aronson_estim_der}, \eqref{eq:estim_mall_der_Gamma_bis} and \eqref{eq:estim_mall_space_der_Gamma_bis} hold for $Z$. In the case $\rma(t) = a(\xi_t)$, the  constants  appearing in inequalities \eqref{eq:estim_mall_der_Gamma_bis} and \eqref{eq:estim_mall_space_der_Gamma_bis} depend on the Lipschitz constant of the matrix $a(y)$. Similar computations  also show that
$$|D_r \partial^2_{x_i x_j} Z(x-y,t,s)| \leq \psi(r)\frac{1}{\ds (t-s)} \ g_{\varsigma,\varpi}(x-y,t-s).$$

We turn to the case of $\rma$ that depends on both $x$ and $t$.
\begin{lemma}[Malliavin differentiability of $\Phi$] \label{lmm:malliavin_Phi}
The function $\Phi$ belongs to $\bD^{1,2}$ for every $(t,s)\in[0,T]^2$, $s< t$ and $(x,y)\in (\R^d)^2$. Moreover, there exists two constants $\varrho> 0$ and $\varpi>0$ such that
\begin{equation} \label{eq:estim_mall_der_Phi}
|D_r \Phi(x,t,y,s)| \leq \psi (r) \frac{1}{\sqrt{t-s}} \ g_{\varrho,\varpi}(x-y,t-s) .
\end{equation}
\end{lemma}
\begin{proof}
Recall that
$$
\gamma_i(x,t) = \mathrm{div} (\rma_i(x,t)) = \sum_{j=1}^d  \frac{\partial \rma_{ji}}{\partial x_j} (x,t).
$$
Note that due to Condition \ref{H5} the process $\gamma$ belongs also to $\bD^{1,2}$. According to \eqref{eq:kernel_first_step} the Malliavin derivative of $\cK$ is given by:
\begin{eqnarray*}
D_r \cK(x,t,y,s)& = & \sum_{i,j=1}^n \left[ D_r \rma_{ij}(x,t)- D_r \rma_{ij}(y,t) \right] \frac{\partial^2 Z}{\partial x_i \partial x_j}(x-y,t,s,y) \\
& + & \sum_{i=1}^n D_r \gamma_i(x,t) \frac{\partial Z^\eps}{\partial x_i } (x-y,t,s,y) \\
& + & \sum_{i,j=1}^n (\rma_{ij}(x,t)-\rma_{ij}(y,t)) D_r \frac{\partial^2 Z}{\partial x_i \partial x_j}(x-y,t,s,y) \\
& + & \sum_{i=1}^n  \gamma_i(x,t) D_r \frac{\partial Z}{\partial x_i } (x-y,t,s,y).
\end{eqnarray*}
From our previous assumptions and properties we deduce that
\begin{eqnarray*}
|D_r \cK(x,t,y,s) |& \leq & \psi(r)   \frac{|x-y|}{\ds (t-s)} \ g_{\varrho,\varpi}(x-y,t-s) \\
& + &  \psi(r) \frac{1}{\ds \sqrt{t-s}} \ g_{\varrho,\varpi}(x-y,t-s) \\
& + &  \frac{ |x-y|}{\ds (t-s)} \ \psi(r) \ g_{\varrho,\varpi}(x-y,t-s) \\
& + &  \psi(r)  \frac{1}{\ds \sqrt{t-s}}  \ g_{\varrho,\varpi}(x-y,t-s) \\
& \leq &  \psi(r) \frac{1}{\ds \sqrt{t-s}} \ g_{\varrho,\varpi}(x-y,t-s).
\end{eqnarray*}
By induction, using the same techniques as in  the proof of Lemma \ref{lem:estim_phi_eps}), we obtain for $m\geq 2$
\begin{eqnarray*}
|D_r \cK_m(x,t,y,s) |& \leq &\frac{M^m}{(1+m/2)!}  \psi(r) (t-s)^{m/2-1} g_{\varrho,\varpi}(x-y,t-s)
\end{eqnarray*}
with some constant  $M > 0$ depending on $\varrho$ and $\varpi$. Indeed, for $m=2$
\begin{eqnarray*}
&& \left| D_r \cK_2(x,t,y,s) \right|  \leq  \int_s^t \int_{\R^d} \left| D_r \cK(x,t,\zeta,\tau) \right| \left| \cK(\zeta,\tau,y,s) \right| d\zeta d\tau \\
 && \qquad +  \int_s^t \int_{\R^d} \left|  \cK(x,t,\zeta,\tau) \right| \left| D_r\cK(\zeta,\tau,y,s) \right| d\zeta d\tau\\
 && \quad \leq 2 \psi(r)  \int_s^t  \int_{\R^d}  \frac{1}{\sqrt{(t-\tau)(\tau-s)} }\ g_{\varrho,\varpi}(x-\zeta,t-\tau)g_{\varrho,\varpi}(\zeta-y,\tau-s) d\zeta d\tau
 \end{eqnarray*}
and the required estimate on the integral can be deduced by the classical arguments.
By the closability of the operator $D$ we conclude that
\begin{equation} \label{eq:def_deriv_Phi}
D_r \Phi(x,t,y,s)  = \sum_{m=1}^\infty D_r \cK_m (x,t,y,s)
\end{equation}
and that estimate \eqref{eq:estim_mall_der_Phi} holds.
Since $\psi(r)$ belongs to $L^2(\Omega)$, $\Phi \in \bD^{1,2}$. This completes the proof of the Lemma.
\end{proof}

We turn to the proof of Theorem \ref{thm:malliavin_deriv_fund_sol}. Let us show that $\Gamma \in \bD^{1,2}$ and that the Gaussian estimates hold for the Malliavin derivative. From the definition of $\Gamma$ in \eqref{eq:def_fund_sol_param_method}, the two previous lemmata and the properties of the Malliavin derivative $D$ we obtain that
\begin{eqnarray} \nonumber
D_\tau \Gamma(x,t,y,s) & = & D_\tau Z(x-y,t,s,y) \\ \nonumber
&+& \int_s^t \int_{\R^d} D_\tau Z(x-\zeta,t,\tau,\zeta) \Phi(\zeta,\tau,y,s) d\zeta d\tau \\ \label{eq:def_mall_deriv_Gamma}
& + & \int_s^t \int_{\R^d} Z(x-\zeta,t,\tau,\zeta) D_\tau  \Phi(\zeta,\tau,y,s) d\zeta d\tau.
\end{eqnarray}
Inequalities \eqref{eq:estim_mall_space_der_Z} and \eqref{eq:estim_mall_der_Phi} imply that
\begin{equation*}
|D_r \Gamma(x,t,y,s)| \leq \psi(r) g_{\varrho,\varpi}(x-y,t-s);
\end{equation*}
for the details see Lemma I.4.3 in \cite{frie:64}. From equation \eqref{eq:def_fund_sol_param_method_deriv} one can obtain an expression for the Malliavin derivative of $\dfrac{\partial}{\partial x_i} \Gamma(x,t,y,s)$:
\begin{eqnarray*} \nonumber
D_r \frac{\partial}{\partial x_i} \Gamma(x,t,y,s) & = & D_r \frac{\partial Z}{\partial x_i}(x-y,t,s,y) \\
& + & \int_s^t \int_{\R^d} D_r \frac{\partial Z}{\partial x_i}(x-\zeta,t,\tau,\zeta) \Phi(\zeta,\tau,y,s) d\zeta d\tau \\
& + & \int_s^t \int_{\R^d} \frac{\partial Z}{\partial x_i}(x-\zeta,t,\tau,\zeta) D_r \Phi(\zeta,r,y,s) d\zeta d\tau.
\end{eqnarray*}
Again with the help of Lemma I.4.3 in \cite{frie:64},  estimates \eqref{eq:estim_mall_space_der_Z} and \eqref{eq:estim_mall_der_Phi} imply  \eqref{eq:estim_mall_space_der_Gamma_bis}. This achieves the proof.

\subsection{Diffusion example} \label{ssect:diff_case}

Here we consider the special case
$\rma(x,t) = a (x,\xi_t),$
with a matrix-valued function $a$ defined on $\R^d \times \R^d$ such that
\begin{itemize}
\item[\bf a1.] $a$ is uniformly elliptic: for any $(x,y,\zeta) \in \R^d \times \R^d \times \R^d$
$$\lambda^{-1} |\zeta|^2\leq a(x,y)\zeta\cdot\zeta\leq \lambda |\zeta|^2.$$
\item[\bf a2.] $a$ is continuous on $\R^d \times \R^d$ and of class $C^1$ w.r.t. $x$ with a bounded derivative: for any $(x,y)$
$$|\nabla_x a(x,y) |  \leq K_a.$$
\end{itemize}
The process $\xi$ is given as the solution of the following SDE:
\begin{equation}\label{eq:SDE}
d\xi_t = \beta(t,\xi_t) dt + \sigma(t,\xi_t) dB_t,
\end{equation}
or, in the coordinate form,
$
d\xi^i_t = \beta_i(t,\xi_t) dt + \sum_{j=1}^d \sigma_{i,j}(t,\xi_t) dB^j_t.
$
We assume that the matrix-function $\sigma$ and vector-function $\beta$ possess the following properties.
\begin{itemize}
\item [\bf c1.]  $\sigma$ and $b$ are globally Lipschitz continuous: there exists $K_{\beta,\sigma}>0$ such that
$$
\|\sigma(t,y')-\sigma(t,y'')\|+ |\beta(t,y')-\beta(t,y'')| \leq K_{\beta,\sigma} |y'-y''|.
$$
\item [\bf c2.] $t \mapsto \sigma(t,0)$ and $t\mapsto \beta(t,0)$ are bounded on $\R_+$.
\item [\bf c3.]  $\sigma$ and $\beta$ are at least two times differentiable w.r.t. $x$ with uniformly bounded derivatives. The absolute value of these derivatives does not exceed a constant that is also denoted by $K_{\beta,\sigma}$.
\end{itemize}
It is well known that under the assumptions {\bf c1} and {\bf c2}, $\xi$ is the unique strong solution of the SDE \eqref{eq:SDE} and for any $T \geq 0$ and any $p\geq 2$
$$
\E \Big( \sup_{t\in [0,T]} |\xi_t|^p \Big) \leq C,
$$
where $C$ is a positive constant depending on $p$, $T$, $K_{\beta,\sigma}$ and $\xi_0$.
The next result can be found in \cite{nual:06}, Theorems 2.2.1 and 2.2.2.
\begin{lemma} \label{prop:malliavin_deriv}
Under conditions {\bf c1}-- {\bf c3}, the coordinate $\xi^i_t$ belongs to $\bD^{1,\infty}$ for any $t\in [0,T]$ and $i=1,\ldots,d$. Moreover for any $j=1,\ldots,d$ and any $p\geq 1$
\begin{equation} \label{eq:estim_deriv_xi_Lp}
\sup_{0\leq r\leq T} \mE \bigg( \sup_{r \leq t \leq T} |D^j_r\xi^i_t |^p \bigg) < +\infty.
\end{equation}
The derivative $D^j_r\xi^i_t$ satisfies the following linear equation:
\begin{eqnarray*}
D^j_r\xi^i_t & = & \sigma_{i,j}(\xi_r) + \sum_{1\leq k,l\leq d} \int_r^t \widetilde \sigma_{i,k}^{l}(s) D^j_r(\xi^k_s) dB^l_s + \sum_{k=1}^d \int_r^t \widetilde b_{i,k}(s) D^j_r (\xi^k_s )ds
\end{eqnarray*}
for $r \leq t$ a.e. and $D^j_r \xi_t = 0$ for $r>t$ a.e., where $\sigma^j$ is the column number $j$ of the matrix $\sigma$ and where for $1\leq i,j\leq d$ and $1\leq l \leq d$, $\widetilde b_{i,j}(s)$ and $\widetilde \sigma_{i,j}^{l}(s)$ are given by:
\begin{equation}\label{eq:coeff_SDE_Mall}
\widetilde b_{i,j}(s) = (\partial_{x_j} b_i) (\xi_s),\qquad \widetilde \sigma_{i,j}^{l}(s) = (\partial_{x_j}  \sigma_{i,l})  (\xi_s).
\end{equation}

The process $\xi$ belongs to $\bD^{2,\infty}$ and the second derivatives $D^i_r D^j_s \xi^k_t$ satisfy also a linear stochastic differential equation with bounded coefficients.
\end{lemma}
For any $r \in [0,T]$ we define
\begin{equation} \label{eq:def_mall_deriv_xi_eps}
\psi(r) = \sup_{t\in [r,T]} \|D_{r} \xi_t\|.
\end{equation}
From Lemma \ref{prop:malliavin_deriv} we have for any $p\geq 2$
\begin{equation} \label{eq:mall_deriv_xi_int_cond}
\sup_{r \in [0,T]} \mE \left( \psi(r)^p \right) < +\infty.
\end{equation}
We define for any $(x,t) \in \R^d \times [0,+\infty)$
$$\rma(x,t) = a(x,\xi_t).$$
Assumptions {\bf a1} and {\bf a2} imply that Conditions \ref{H1}, \ref{H2} and \ref{H3} hold. Moreover let us assume that the matrix $a$ is smooth w.r.t. $y$ and satisfies the following regularity conditions.
\begin{description}
\item[\bf a3.] For any $1\leq j,k\leq d$
$$\left| \nabla_y a (x,y)\right| + \left| \frac{\partial^2}{\partial x_j \partial y_k} a (x,y)\right| \leq K_a.$$
\end{description}

Using conditions {\bf a2} and {\bf a3}, the previous lemma and the classical chain rule (see Proposition 1.2.3 in \cite{nual:06}), we obtain that
$$D^j_r \rma_{i,\ell}(x,t) = \sum_{k} \frac{\partial a_{i,\ell}}{\partial y_k} (x,\xi_t) D^j_r  \xi^k_t .$$
Thus $D_r \rma(x,t) = 0$ if $r > t$,  while for $r \leq t$ we have
\begin{equation} \label{eq:control_der_a_eps}
|D^k_r \rma_{ij}(x,t)| \leq  \left| \frac{\partial a_{ij}}{\partial y_\ell}  \right|  |D_r (\xi_{t}) | \leq K_a \psi(r).
\end{equation}
The same computation shows that
\begin{eqnarray*}
&& D^k_r \frac{\partial \rma_{ij}}{\partial x_\ell} (x,t)  =\sum_{k} \frac{\partial^2 a_{i,\ell}}{\partial x_\ell \partial y_k} (x,\xi_t) D^j_r  \xi^k_t .
\end{eqnarray*}
Hence
$$
\Big|D^k_r \frac{\partial \rma_{ij}}{\partial x_\ell}(x,t)\Big|  \leq K_a \psi(r).
$$
We deduce that $\rma(x,t)$ belongs to $\bD^{1,\infty}$ (condition \ref{H4}), the previous computations yield \ref{H5}, and $\psi$ satisfies the integrability condition \eqref{eq:mall_deriv_xi_int_cond}. From Theorems \ref{thm:general_result_Aronson_estim} and \ref{thm:malliavin_deriv_fund_sol} we deduce immediately the following result.
\begin{coro} \label{coro:aronson_estim_diff_case}
Under assumptions {\bf a1}\,--\,{\bf a3} on the matrix $a$ and conditions {\bf c1}\,--\,{\bf c3} on the coefficients of the SDE \eqref{eq:SDE}, if $\rma(x,t) = a(x,\xi_t)$, then the fundamental solution $\Gamma$ of equation \eqref{eq:PDE_fund_sol} and its spatial derivatives belong to $\bD^{1,2}$ and satisfy Estimates \eqref{eq:aronson_estim}, \eqref{eq:aronson_estim_der}, \eqref{eq:estim_mall_der_Gamma_bis} and \eqref{eq:estim_mall_space_der_Gamma_bis}.
\end{coro}

\section{Mild solution of the heat SPDE} \label{sect:mild_sol}

In this last section we  construct a mild solution $v$ to the heat SPDE \eqref{eq:SPDE}
 with the initial condition $v(x,0)=0$, that is we construct a solution $v$ of equation \eqref{eq:mild_sol}.
\begin{remark}\label{rem:initial_cond}
If the initial condition for $v$ is given by a function $\imath$, then by linearity of the SPDE, we should add in \eqref{eq:mild_sol} one term:
$$v(x,t)= \int_0^t \int_{\R^d} \Gamma(x,t,y,s) G\left(y,s \right)dy dB_{s} + \int_{\R^d} \Gamma(x,t,y,0) \imath(y) dy$$
Under the setting of Theorem \ref{thm:general_result_Aronson_estim}, this additional term is well defined provided that the function $\imath$ increases no faster than a function $\exp(cx^2)$ (see \cite[Theorem I.7.12]{frie:64}).
\end{remark}
Let us specify our setting. We still assume that all hypotheses \ref{H1} to \ref{H5} hold and we add several conditions on $G$.
\begin{enumerate}[label=\textbf{(D\arabic*)}]
\item \label{D1} The function $G : \R^d \times [0,+\infty) \times \Omega \to \R^d$ is a progressively measurable function  that
satisfies the estimate $(1+|x|)^N |G(x,t)|\leq \fG (t)$
for some  $N > d/2$ and some adapted process $\fG$ such that
$$
\mE \left( \int_0^T \fG(t)^{2q} dt \right) < +\infty
$$
with some $q>1$.
\item \label{D2}For each $(x,t) \in \R^d \times [0,+\infty)$, the random variable $G(x,t)$ belongs to $\bD^{1,2}$, and for any $t \in [0,T]$ and any $x \in \R^d$,
$$|D_r G(x,t)|  \leq \widetilde G(x,t) \psi(r).$$
The process $\psi$ is the same as in Condition \ref{H5} and $\widetilde G$ verifies the growth assumption:
$(1+|x|)^N |\widetilde G(x,t)|\leq \fG (t).$
\item \label{D3} The constants $p$ of \ref{H5} and $q$ verify: $p > q > 2d+4.$
\item \label{D4} The process $\fG$ verifies $\displaystyle \mP \left( \sup_{t\in [0,T]} \fG(t) < +\infty \right) = 1$.
\end{enumerate}

\begin{remark} \label{rem:comments_hyp_D}

Under {\rm \ref{D3}}, we have the weaker condition $\frac{1}{p} + \frac{1}{q} \leq 1$. From the proofs, we are aware that this condition {\rm \ref{D3}} is a little bit too strong. But a relation between $p$, $q$ and $d$ is needed with our arguments. In \cite{alos:leon:nual:99}, this relation is implicitly given: for example in Theorem 3.5, the authors impose $p>8$ (for $d=1$). {\rm \ref{D1}} and {\rm \ref{D4}} is a little bit more general than in \cite{alos:leon:nual:99} where $G$ is bounded with respect to $(x,t)$.

Moreover the following relations hold:
$$
2 \leq \kappa \leq  \frac{2pq}{p+q}  \Rightarrow  \frac{q}{q-1} \leq \frac{\kappa q}{2q-\kappa} \leq p,\quad
\frac{1}{p} + \frac{1}{q} \leq 1  \Leftrightarrow  \frac{2pq}{p+q}\geq 2,
$$
and
$$\frac{1}{2} + \frac{p+q}{2pq} \leq \frac{q-1}{q} \leq 1 - \frac{p+q}{2pq} \leq \frac{2p-1}{2p}.$$
\end{remark}

Let us give our third main result.
\begin{thm} \label{thm:reg_mild_sol}
Let assumptions {\rm \ref{H1}}\,--\,{\rm \ref{H5}} be fulfilled, and assume that  conditions {\rm \ref{D1}}\,--\,{\rm \ref{D4}} hold. Then on $\R^d \times (0,+\infty)$, the random field $v$ given by \eqref{eq:mild_sol} is well defined, is continuous w.r.t. $(x,t)$ and has first derivatives w.r.t. $x$  such that
$$\mE \left[ \sup_{x,t}  \left( | v(x,t) |^{\frac{2pq}{p+q}} +  |\nabla v(x,t) |^{\frac{2pq}{p+q}} \right) \right] < +\infty.$$
Moreover 
$v$ is a weak solution of  the SPDE in \eqref{eq:SPDE}.
\end{thm}
The notion of a weak solution is explained in Definition \ref{def:weak_sol}. 

\subsection{The diffusion case}

Again we assume that $\rma(x,t)=a(x,\xi_t)$ where $\xi$ is the solution of the SDE in \eqref{eq:SDE}. Let us fix a measurable function $g :\R^d \times [0,+\infty) \times \R^d \to \R^d$ such that $g$ is of class $C^1$ w.r.t. the last component and
\begin{equation*}
G\left(x,t \right)=  g(x,t,\xi_{t})
\end{equation*}
Then the Malliavin derivative of $G$ can be computed by a chain rule argument: $D_rG\left(x,t \right) = \nabla_y g(x,t,\xi_t ) D_r \xi_{t}$. Hence
$$|D_rG\left(x,t \right) | \leq | \nabla_y g(x,t,\xi_t )| \psi(r).$$
Let us assume that for some $N>d/2$:
$$|g(x,t,y)| + |\nabla_y g(x,t,y)| \leq C\frac{|y|}{(1+|x|)^N}.$$
Then $\fG(t) = |\xi_t|$ is continuous w.r.t. $t$, thus \ref{D4} holds. And, for any $q > 1$,
$$
\E \left( \sup_{t\in [0,T]} |\xi_t|^{2q} \right) \leq C.
$$
Therefore, \ref{D1} and \ref{D3} are also satisfied. From Theorem \ref{thm:reg_mild_sol} we get
\begin{coro} \label{coro:reg_mild_sol}
Under conditions {\bf a1}\,--\,{\bf a3} on the matrix $a$ and {\bf c1}\,--\,{\bf c3} on the coefficients of the SDE, if the previous assumptions are satisfied, then the conclusion of Theorem \ref{thm:reg_mild_sol} holds in the diffusion case.
\end{coro}

\subsection{Construction of the mild solution} \label{ssect:mild_sol_cons}

The rest of the paper is devoted to the proof of Theorem \ref{thm:reg_mild_sol}. Let us first specify the meaning of a weak solution of equation \eqref{eq:SPDE}. 
\begin{defin} \label{def:weak_sol}
Let $v=\{ v(x,t), \ (x,t) \in \R^d \times [0,+\infty)\}$ be a random field. We say that $v$ is a weak solution of equation \eqref{eq:SPDE} if
\begin{itemize}
\item $v$ is continuous on $\R^d \times (0,+\infty)$. Moreover,  a.s. for any $x\in \R^d$,
$$\lim_{t \downarrow 0} v(x,t) = 0\ ;$$
\item $v$ has all first order partial derivatives in $x$ on $\R^d \times (0,+\infty)$ ;
\item for any test function $\phi \in C^\infty_0(\R^d)$ and for all $t \in [0,T]$ we have
\!\!\!\begin{eqnarray*} \int\limits_{\R^d}\! v(x,t) \phi(x) dx+ \!\!  \int\limits_0^t\!\! \int\limits_{\R^d}\!  a(x,s) \nabla \phi(x)  \nabla  v(x,s) dx  =\!\! \int\limits_0^t  \!\!\int\limits_{\R^d}\! G(x,s) \phi(x) dx dB_s.
\end{eqnarray*}
\end{itemize}
\end{defin}
Our aim is to prove that the random function $v$ given by \eqref{eq:mild_sol} is a weak solution of the SPDE \eqref{eq:SPDE}. The stochastic integral in \eqref{eq:mild_sol} has to be defined properly since $\Gamma(x,t,y,s)$ is measurable w.r.t. the $\sigma$-field $\mathcal{F}_{t}$ generated by the random variables $B_u$ with $u \leq t$. The correct definition can be found in \cite{nual:pard:88} and is based on Malliavin's calculus.
To define a mild solution of  \eqref{eq:SPDE}, let us recall \cite[Definition 3.1]{nual:pard:88}.
\begin{defin} \label{def:L_1_2}
Let $\bL^{1,2}$ denote class of scalar processes $u \in \bL^2([0,T]\times \Omega)$ such that $u_t \in \bD^{1,2}$ for a.a. $t$ and there exists a measurable version of $D_r u_t$ verifying
$$\mE \int_0^T \int_0^T |D_r u_t|^2 dr dt < +\infty.$$
$\bL^{1,2}_d$ is the set of $d$-dimensional processes whose components are in $\bL^{1,2}$.
\end{defin}
\begin{prop} \label{prop:L_2_estimate_on_v_2}
For any $(t,x) \in [0,T] \times \R^d$, the stochastic integral
$$v(x,t)= \int_0^t \int_{\R^d} \Gamma(x,t,y,s) G\left(y,s \right)dy dB_{s}$$
is well defined and
$$\mE \left[ \int_0^T \int_{\R^d} (v(x,t))^2 dx dt \right] < +\infty.$$
\end{prop}
\begin{proof}
From Theorem \ref{thm:malliavin_deriv_fund_sol} and  condition \ref{D1} on $G$, we deduce that for each $(x,t) \in \R^d \times [0,T]$, the process
\begin{equation}\label{eq:def_proc_u}
u(x,t,s) =  \int_{\R^d} \Gamma(x,t,y,s) G(y,s) dy
\end{equation}
is well defined. Aronson's estimate \eqref{eq:aronson_estim}, H\"older's inequality and condition \ref{D1} lead to:
\begin{equation} \label{eq:upper_bound_u}
|u(x,t,s)|^2 \leq C   \int_{\R^d} g_{\varsigma,\varpi}(x-y,t-s) |G(y,s)|^2 dy \leq \frac{C}{(1+|x|)^{2N}}\fG(s)^2  .
\end{equation}
Therefore,
\begin{equation} \label{eq:L2_cond_anticipating_int}
\mE \int_0^t |u(x,t,s)|^2 ds \leq \frac{C^2}{(1+|x|)^{2N}}\mE \int_0^t \fG(s)^2 ds < +\infty.
\end{equation}
Moreover,
$$D_r u(x,t,s)  = \int_{\R^d} \left[ D_r \Gamma(x,t,y,s) G(y,s) + \Gamma(x,t,y,s) D_r G(y,s) \right] dy.$$
Therefore, from estimate \eqref{eq:estim_mall_der_Gamma_bis} on $D_r \Gamma$, H\"older's inequality and conditions \ref{D1} and \ref{D2} for $G$ and $\widetilde G$, we obtain
\begin{eqnarray}  \nonumber
|D_r u(x,t,s)|^2 & \leq &\left(  \int_{\R^d} \left| D_r \Gamma(x,t,y,s) G(y,s) + \Gamma(x,t,y,s) D_r G(y,s) \right| dy\right)^2 \\  \nonumber
& \leq & \psi(r)^2 C  \int_{\R^d} g_{\varrho,\varpi}(x-y,t-s) \left[ |G(y,s)|^2 + |\widetilde G (y,s)|^2 \right] dy \\\label{eq:upper_bound__mall_der_u}
& \leq &   \frac{C}{(1+|x|)^{2N}} \psi(r)^2\fG (s)^2.
\end{eqnarray}
Applying again the H\"older inequality yields
\begin{eqnarray*}
&& \mE \int_0^t \int_0^t |D_r u(x,t,s) |^2 ds dr \\ \nonumber
&& \quad  \leq  \frac{C^2}{(1+|x|)^{2N}}\mE\left[ \left(  \int_0^t \psi(r)^2 dr \right)^{\frac{q}{q-1}}\right]^{\frac{q-1}{q}}  \mE \left[ \left( \int_0^t \fG (s)^2 ds \right)^{q} \right]^{\frac{1}{q}}.
\end{eqnarray*}
Since $p \geq q/(q-1)$, using Jensen's inequality, we obtain
\begin{equation}\label{eq:L2_cond_anticipating_int_2}
 \mE \int_0^t \int_0^t |D_r u(x,t,s) |^2 ds dr < +\infty.
\end{equation}
Conditions \eqref{eq:L2_cond_anticipating_int} and \eqref{eq:L2_cond_anticipating_int_2} are exactly the ones required in Definition \ref{def:L_1_2}. Hence $u(x,t,s)$ belongs to the space $\bL^{1,2}_d$ and the stochastic integral $v(x,t)$
is well-defined for any $(x,t)$. Moreover, the isometric property of the anticipating It\^o integral holds (see Eq. (3.5) in \cite{nual:pard:88}):
\begin{eqnarray*}
\mE((v(x,t))^2) & = & \mE \int_{0}^t |u(x,t,s)|^2 ds + \mE \int_{0}^t\int_{0}^t  |D_r u (x,t,s)|^2 ds dr.
\end{eqnarray*}
From our previous estimates \eqref{eq:L2_cond_anticipating_int} and \eqref{eq:L2_cond_anticipating_int_2}, we obtain that
\begin{eqnarray*}
&& \mE \int_0^T \int_{\R^d} (v(x,t))^2 dx dt \\
&&\quad  \leq C \mE\left[ \left(  \int_0^T \psi(r)^2 dr \right)^{\frac{q}{q-1}}\right]^{\frac{q-1}{q}}  \mE \left[ \left( \int_0^T \fG(s)^2 ds \right)^{q} \right]^{\frac{1}{q}}.
\end{eqnarray*}
\end{proof}

We are going to prove that $(x,t)\mapsto v(x,t)$ is continuous and $x \mapsto v(x,t)$ is differentiable. Note that we cannot directly use \cite[Theorem 5.2]{nual:pard:88} since $\Gamma$ also depends on $t$. Even if $\Gamma$ is continuous on $\{0\leq s < t \leq T\}$, the singularity at time $t$ should be handled carefully. We follow some ideas contained in \cite[Section 3]{alos:leon:nual:99} and the regularity results concerning the volume potential (see Lemmata \ref{lem:reg_volume_potential} and \ref{lem:reg_volume_potential_2} in the Appendix). The main trick is to transform the anticipating stochastic integral $v$ into a Lebesgue integral. 

\subsubsection{Another representation of $v$}

Given $\alpha \in (0,1)$  define for any $(t,x) \in [0,T]\times \R^d$:
\begin{eqnarray}  \label{eq:def_X}
X(x,t) &=& \int_0^t \int_{\R^d} (t-s)^{-\alpha} D_s \Gamma(x,t,y,s) G(y,s) dy ds,\\  \label{eq:def_Y}
Y(x,t) &=& \int_0^t \int_{\R^d} (t-s)^{-\alpha} \Gamma(x,t,y,s)G(y,s)dy dB_s.
\end{eqnarray}
Due to the Aronson estimate \eqref{eq:estim_mall_der_Gamma_bis} on $D_s \Gamma$ and hypothesis \ref{D1} on $G$
the field $X$ is well defined for any $\al \in [0,1)$.
\begin{lemma} \label{lem:properties_of_X}
Assume that $0\leq  \alpha < \frac{2p-1}{2p}$. 
Then a.s. $(x,t) \mapsto X(x,t)$ is continuous. Moreover, for any $\alpha < 1-\frac{p+q}{2pq}$ and any $1 < \delta \leq \frac{2pq}{p+q}$
\begin{eqnarray*}
&& \mE \left(\sup_{x,t}  |X(x,t)|^{\delta} \right) +\mE \int_0^T \int_{\R^d} |X(x,t)|^{\delta} dx dt  \\
&& \qquad \leq C \mE \left[ \left( \int_0^{T} \psi(s)^{2p} ds \right)\right]^{\frac{q}{p+q}} \left[ \mE \left( \int_0^{T} \fG(s)^{2q} ds \right)\right]^{\frac{p}{p+q}}.
\end{eqnarray*}
Assume furthermore that $0\leq \alpha < \frac{p-1}{2p}$. Then a.s. $x \mapsto X(x,t)$ is differentiable:
$$\nabla X(x,t) = \int_0^{t} \int_{\R^d} (t-s)^{-\alpha} D_s \nabla \Gamma(x,t,y,s) G(y,s) dy ds$$
and if $0\leq 2\alpha < 1-\frac{p+q}{pq}$, then $\displaystyle \mE \left[ \sup_{x,t} |\nabla X(x,t)|^{\frac{2pq}{p+q}} \right] < +\infty.$
\end{lemma}
\begin{proof}
We already know that $D_s \Gamma(x,t,y,s)$ is continuous w.r.t. $(x,y)$ and $s < t$. Thanks to \eqref{eq:estim_mall_der_Gamma_bis}, we have a.s.
\begin{eqnarray*}
&&  \int_{\R^d} (t-s)^{-\alpha} |D_s \Gamma(x,t,y,s)| dy  \\
&& \quad \leq  \int_{\R^d} (t-s)^{-\alpha} \psi(s) g_{\varrho,\varpi}(x-y,t-s) dy  \leq C  \psi(s) (t-s)^{-\alpha}.
\end{eqnarray*}
From our assumption on $\alpha$ and $\psi$ we have
$$\int_0^t \psi(s) (t-s)^{-\alpha} ds \leq  \left( \int_{0}^t \psi(s)^{2p} ds \right)^{\frac{1}{2p}}  \left( \int_{0}^t (t-s)^{-\frac{2p\alpha}{2p-1}} ds \right)^{\frac{2p-1}{2p}}< +\infty.$$
Moreover,  a.s.
$$\sup_{y,s} |G(y,s)| \leq \sup_{y,s} \frac{\fG(s)}{(1+|y|)^N} < +\infty.$$
Arguing as in the proof of Lemma \ref{lem:reg_volume_potential}, we get the a.s. continuity of $X$ w.r.t. $(x,t)$. From estimate \eqref{eq:estim_mall_der_Gamma_bis} on $D_s \Gamma$ we deduce
$$ \int_{\R^d}  \left| D_s \Gamma(x,t,y,s) G(y,s) \right| dy  \leq C \psi(s) \fG(s)  \frac{1}{(1+|z|)^{N}}.$$
Let us choose $r > 1$ such that $1/r+1/(2p)+1/(2q) = 1$ and $\alpha r <1$. Then
\begin{eqnarray*}
&& |X(x,t)|  \leq  \frac{C}{(1+|x|)^{N}} \int_0^t (t-s)^{-\alpha} \psi(s) \fG(s) ds \\
&&  \ \leq \frac{C}{(1+|x|)^{N}} \left(  \int_0^t (t-s)^{- r\alpha} ds \right)^{\frac{1}{r}} \left( \int_0^t  \fG(s)^{2q} ds \right)^{\frac{1}{2q}}\left(  \int_0^t \psi(s)^{2p} ds \right)^{\frac{1}{2p}} \\
&& \ \leq  \frac{C}{(1+|x|)^{N}} \left( \int_0^T  \fG(s)^{2q} ds \right)^{\frac{1}{2q}}\left(  \int_0^T \psi(s)^{2p} ds \right)^{\frac{1}{2p}}.
\end{eqnarray*}
Finally,  the H\"older and Jensen inequalities lead to the desired result.

To obtain the differentiability observe that estimate \eqref{eq:estim_mall_space_der_Gamma_bis} leads to:
\begin{eqnarray*}
(t-s)^{-\al}|D_s \nabla_x \Gamma(x,t,y,s) | & \leq&  \psi(s)(t-s)^{-\al-1/2} g_{\varrho,\varpi}(x-y,t-s) .
\end{eqnarray*}
It then remains to apply the same arguments as above with $\al+1/2$ instead of $\alpha$.
%
\end{proof}

In the next lemma we prove that $Y$ is well defined and integrable.
\begin{lemma} \label{lem:auxillary_alpha_process}
For any $(t,s,x) \in [0,T]^2\times \R^d$ and any $0\leq \alpha < \frac{q-1}{q}$, the process
\begin{eqnarray*}
u_\al(x,t,s)&=&(t-s)^{-\alpha} \int_{\R^d}\Gamma(x,t,y,s)G(y,s) dy \mathbf 1_{[0,t)}(s)
\end{eqnarray*}
belongs to $\bL^{1,2}_d$.  
Moreover, for any $2 \leq \kappa \leq \frac{2pq}{p+q}$ it holds
\begin{eqnarray*}
&& \mE \left[  \left| Y(x,t) \right|^{\kappa} \right] \leq  \frac{C}{(1+|x|)^{\kappa N}} .
\end{eqnarray*}
\end{lemma}
\begin{proof}
As was shown in the proof of Proposition \ref{prop:L_2_estimate_on_v_2}, we have the upper bound \eqref{eq:upper_bound_u} on $u$ and \eqref{eq:upper_bound__mall_der_u} on $D_r u$.
Thus by the H\"older inequality
\begin{eqnarray*}
&&\mE \int_0^t (t-s)^{-2\alpha} \left| \int_{\R^d} \Gamma(x,t,y,s)G(y,s) dy  \right|^2 ds  \\
&& \quad \leq \frac{C}{(1+|x|)^{2N}} \left(  \int_0^t (t-s)^{- \frac{\alpha q}{q-1}} ds \right)^{\frac{q-1}{q}} \left( \mE \int_0^t  \fG(s)^{2q} ds \right)^{\frac{1}{q}} < +\infty;
\end{eqnarray*}
here we have also used the inequality $\frac{\alpha q}{q-1} < 1$. Similarly,
\begin{eqnarray*}
&& \mE \int_0^t \int_0^t  (t-s)^{-2\alpha} \left| D_r u(x,t,s)  \right|^2 drds \\
&& \quad \leq \frac{C}{(1+|x|)^{2N}}   \left[\mE\left(  \int_0^t \fG(s)^{2q} ds \right)\right]^{\frac{1}{q}} \left[\mE \left(  \int_0^t \psi(r)^2 dr \right)^{\frac{q}{q-1}}\right]^{\frac{q-1}{q}}.
\end{eqnarray*}
Since $p\geq q/(q-1)$, by the Jensen inequality we derive that the process $u_\al$ is in $\bL^{1,2}_d$. Now, using \cite[Proposition 3.5]{nual:pard:88}, we have for any $ \kappa\geq 2$
\begin{eqnarray*}
&& \mE \left(  \left| \int_0^t \int_{\R^d} (t-s)^{-\alpha} \Gamma(x,t,y,s)G(y,s) dy dB_s \right|^{\kappa} \right) \\
&&\quad \leq c_{\kappa} \left( \int_0^t (t-s)^{-2\alpha} \left| \mE (u(x,t,s))  \right|^2 ds \right)^{\kappa/2} \\
&& \quad + c_{\kappa}\  \mE \left[ \left( \int_0^t \int_0^t  (t-s)^{-2\alpha} \left|  D_r u(x,t,s) \right|^2 drds\right)^{\kappa/2}\right].
\end{eqnarray*}
Combining this with the  previous inequalities we get
\begin{eqnarray*}
&& \mE \left( \left| \int_0^t \int_{\R^d} (t-s)^{-\alpha} \Gamma(x,t,y,s)G(y,s) dy dB_s \right|^{\kappa} \right)\\
&&\quad \leq  \frac{C}{(1+|x|)^{\kappa N}} \left( \mE \int_0^t  \fG(s)^{2q} ds \right)^{\frac{\kappa}{2q}} \\
&& \quad +\frac{C}{(1+|x|)^{\kappa N}} \mE \left[\left(  \int_0^t \fG(s)^{2q} ds \right)^{\frac{\kappa}{2q}}\left(  \int_0^t \psi(r)^2 dr \right)^{\frac{\kappa}{2}}\right]  \\
&&\quad \leq \frac{C}{(1+|x|)^{\kappa N}} \left[ \mE \int_0^T \fG(s)^{2q} ds \right]^{\frac{\kappa}{2q}} \left\{ 1 + \left[  \mE\left(  \int_0^T \psi(r)^2 dr \right)^{\frac{\kappa q}{2q-\kappa}} \right]^{\frac{2q-\kappa}{2q}}\right\}.
\end{eqnarray*}
This gives the conclusion of the lemma.
\end{proof}

In particular if $N>d/2$, the process $Y$ belongs to $L^\kappa([0,T]\times \R^d \times \Omega)$.
We use the semigroup property of the fundamental solution to derive the desired representation of $v$.
\begin{lemma}
For any $0 < \alpha < \frac{q-1}{q}$, $v(x,t)$ admits the following representation:
\begin{eqnarray}  \nonumber
v(x,t) & = & \frac{\sin(\pi \alpha)}{\pi} \int_0^t \int_{\R^d}(t-r)^{\alpha-1}\Gamma(x,t,z,r)(Y(z,r)+ X(z,r))  dz dr \\ \label{eq:v_representation}
& - &  \int_0^t \int_{\R^d} D_s \Gamma(x,t,y,s) G(y,s) dy ds,
\end{eqnarray}
where $X$ and $Y$ are given by \eqref{eq:def_X} and \eqref{eq:def_Y}.
\end{lemma}
\begin{proof}
Recall that for any $\alpha \in (0,1)$,
$$\Gamma(x,t,y,s)= \frac{\sin(\pi \alpha)}{\pi} \int_s^t \int_{\R^d} (t-r)^{\alpha-1} (r-s)^{-\alpha} \Gamma(x,t,z,r)\Gamma(z,r,y,s) dz dr.$$
Applying Fubini's theorem for the Skorohod integral we obtain
\begin{eqnarray*}
&& v(x,t) = \int_0^t \int_{\R^d} \Gamma(x,t,y,s) G(y,s) dy dB_s \\
&& \quad = \frac{\sin(\pi \alpha)}{\pi} \int_0^t \int_{\R^d} \Bigg[ \int_0^r \int_{\R^d} (t-r)^{\alpha-1} \Gamma(x,t,z,r)  \\
&&\hspace{5cm}  (r-s)^{-\alpha} \Gamma(z,r,y,s) G(y,s) dy dB_s  \Bigg] dz dr.
\end{eqnarray*}
By Lemma \ref{lem:auxillary_alpha_process} with $0 < \alpha < \frac{q-1}{q}$ and $u_\al(r,x,s)\in \bL^{1,2}_d$, and by \cite[Theorem 3.2]{nual:pard:88} we have
\begin{eqnarray*}
&& \int_0^r \int_{\R^d} (t-r)^{\alpha-1}\Gamma(x,t,z,r)(r-s)^{-\alpha} \Gamma(z,r,y,s) G(y,s) dy dB_s \\
&&  \quad =(t-r)^{\alpha-1}\Gamma(x,t,z,r) Y(z,r) \\
&& \quad - \int_0^r \int_{\R^d}(t-r)^{\alpha-1} D_s \Gamma(x,t,z,r) (r-s)^{-\alpha} \Gamma(z,r,y,s) G(y,s) dy ds.
\end{eqnarray*}
Hence
\begin{eqnarray*}
&& v(x,t)  = \frac{\sin(\pi \alpha)}{\pi} \int_0^t \int_{\R^d}(t-r)^{\alpha-1}\Gamma(x,t,z,r) Y(z,r)  dz dr \\
&& \quad - \frac{\sin(\pi \alpha)}{\pi} \int_0^t \int_{\R^d}  (t-r)^{\alpha-1}  \Bigg[ \int_0^r\int_{\R^d} D_s \Gamma(x,t,z,r) \\
&& \hspace{5cm} (r-s)^{-\alpha} \Gamma(z,r,y,s) G(y,s) dy ds \Bigg] dz dr .
\end{eqnarray*}
Since for $0 \leq s < r < t \leq T$ we have
\begin{eqnarray*}
D_s \Gamma(x,t,y,s) & = & \int_{\R^d} \left[ D_s \Gamma(x,t,z,r)\Gamma(z,r,y,s) + \Gamma(x,t,z,r)D_s \Gamma(z,r,y,s) \right] dz ,
\end{eqnarray*}
then
\begin{eqnarray*}
&& v(x,t)= \frac{\sin(\pi \alpha)}{\pi} \int_0^t \int_{\R^d}(t-r)^{\alpha-1}\Gamma(x,t,z,r) (Y(r,z) + X(r,z))  dz dr \\
&& \ - \frac{\sin(\pi \alpha)}{\pi} \int_0^t  (t-r)^{\alpha-1}  \Bigg[ \int_0^r  \int_{\R^d}  (r-s)^{-\alpha}  D_s \Gamma(x,t,y,s) G(y,s) dy ds \Bigg] dr
\end{eqnarray*}
By the Fubini theorem we deduce the representation \eqref{eq:v_representation}.
\end{proof}


\subsubsection{Regularity of the process $v$}

Now we assume that $0< \al < \frac{q-1}{q}$ and study separately the three terms in the decomposition \eqref{eq:v_representation} of $v$. Let us begin with the last one, namely
$$I_3(x,t) =  \int_0^t \int_{\R^d} D_s \Gamma(x,t,y,s) G(y,s) dy ds.$$
Remark that $I_3$ is equal to $X$ with $\alpha = 0$. By Lemma \ref{lem:properties_of_X} with $\alpha=0$, a.s. the mapping $(x,t) \mapsto I_3(x,t)$ is continuous, $x \mapsto I_3(x,t)$ is differentiable, and
$$
\mE \left[ \sup_{x,t} \left( |I_3(x,t)|^{\frac{2pq}{p+q}}  + |\nabla I_3(x,t)|^{\frac{2pq}{p+q}}\right) \right]< +\infty.
$$
We proceed with the term $I_2$ given by
$$
I_2(x,t) = \int_0^{t} \int_{\R^d}(t-r)^{\alpha-1}\Gamma(x,t,z,r) X(z,r)  dz dr.
$$
Notice that for all  $p>q$ we have $ 1-\frac{p+q}{2pq} > (q-1)/q$.
Therefore, for $\alpha < \frac{q-1}{q}$ by Lemma \ref{lem:properties_of_X} we obtain
\begin{equation}\label{estiforx}
\mE \left(\sup_{x,t}  |X(x,t)|^{\frac{2pq}{p+q}} \right) < +\infty.
\end{equation}
Thus a.s. $X$ is bounded w.r.t. $(x,t)$. Arguing as in the proof of Lemma \ref{lem:properties_of_X}, we show that for all $\alpha$ such that $1/2 + \frac{p+q}{2pq} < \alpha<\frac{q-1}{q}$ the term $I_2$ has the same regularity as $I_3$ with
$$
\nabla I_2(x,t) = \int_0^{t} \int_{\R^d}(t-r)^{\alpha-1} \nabla \Gamma(x,t,z,r) X(z,r)  dz dr.
$$
Up to now the dimension $d$ plays no role in our estimate, and we only used \ref{D1}, \ref{D2} and the relation $p> q > 4$. To control $I_2$, we used the fact that $\sup_{x,t} |X(x,t)|$ is a.s. finite.  The estimate in the next statement does depend  on $d$. Remark that if $p>q > 2d+4$, then
$$\frac{1}{2} + \frac{(d+2)(p+q)}{4pq} < \frac{q-1}{q}.$$

\begin{lemma}
Assume that $\frac{1}{2} + \frac{(d+2)(p+q)}{4pq} < \alpha < \frac{q-1}{q}$. Then 
$$\mE \left[\sup_{x,t} | I_2(x,t) |^{\frac{2pq}{(p+q)}} + \sup_{x,t} |\nabla I_2(x,t) |^{\frac{2pq}{(p+q)}} \right] < +\infty.$$
\end{lemma}
\begin{proof}
We only detail the arguments for the gradient of $I_2$ ; for $I_2$ itself they are similar. Note that $\frac{1}{2} + \frac{(d+2)(p+q)}{4pq} < \alpha$ is equivalent to $\left(\alpha-\frac{3}{2}-\frac{d}{2\delta} \right)\frac{\delta}{\delta-1} > -1$ with $\delta = \frac{2pq}{(p+q)}$. Thus by the H\"older inequality:
%
\begin{eqnarray*}
&& \left| \int_0^{t} \int_{\R^d}(t-r)^{\alpha-1} \nabla \Gamma(x,t,z,r)X(z,r)  dz dr \right| \\
&& \ \leq \int_0^{t} (t-r)^{\alpha-1} \left( \int_{\R^d} |\nabla \Gamma(x,t,z,r)|^{\frac{\delta}{\delta-1}} dz \right)^{\frac{\delta-1}{\delta}} \left( \int_{\R^d}  |X(z,r)|^{\delta} dz \right)^{\frac{1}{\delta}} dr.
\end{eqnarray*}
From Estimate \eqref{eq:aronson_estim_der}, we obtain
\begin{eqnarray*}
| \nabla_x \Gamma(x,t,y,r) |^{\frac{\delta}{\delta-1}} &\leq &(t-r)^{-\delta/2(\delta-1)} \ g_{\varrho,\varpi}(x-y,t-r)^{\frac{\delta}{\delta-1}}\\
& \leq & (t-r)^{-\frac{\delta}{2(\delta-1)} -\frac{d}{2}\frac{1}{\delta-1}} \ g_{\varrho',\varpi'}(x-y,t-r)
\end{eqnarray*}
with $(\varrho',\varpi') = (\varrho^{\frac{\delta}{\delta-1}},\varpi \frac{\delta}{\delta-1})$. This yields
$$
(t-r)^{\alpha-1} \left( \int_{\R^d} |\nabla \Gamma(x,t,z,r)|^{\frac{\delta}{\delta-1}} dz\right)^{\frac{\delta-1}{\delta}}  \leq C(t-r)^{\alpha-1} (t-r)^{-\frac{1}{2} -\frac{d}{2\delta}} .
$$
Using again the H\"older inequality we arrive at the estimate
\begin{eqnarray*}
&&\int_0^{t} (t-r)^{\alpha-1} \left( \int_{\R^d} |\nabla \Gamma(x,t,z,r)|^{\frac{\delta}{\delta-1}} dz \right)^{\frac{\delta-1}{\delta}} \left( \int_{\R^d}  |X(z,r)|^{\delta} dz \right)^{\frac{1}{\delta}} dr \\
&& \  \leq C \left(  \int_0^{t} (t-r)^{(\alpha-\frac{3}{2}-\frac{d}{2\delta})\frac{\delta}{\delta-1}} dr \right)^{\frac{\delta-1}{\delta}} \left( \int_0^{t} \int_{\R^d}  |X(z,r)|^{\delta} dz dr \right)^{\frac{1}{\delta}} \\
&& \  \leq C  \left( \int_0^{T} \int_{\R^d}  |X(z,r)|^{\delta} dz dr \right)^{\frac{1}{\delta}}.
\end{eqnarray*}
Thereby
$$
 \sup_{x,t} \left| \nabla I_2(x,t) \right|^{\frac{2pq}{(p+q)}}  \leq C  \left( \int_0^{T} \int_{\R^d}  |X(z,r)|^{\frac{2pq}{(p+q)}} dz dr \right).
 $$
Taking the expectation and considering \eqref{estiforx} we obtain the desired statement.
\end{proof}

It remains to estimate the term $I_1$ in decomposition \eqref{eq:v_representation}. It reads
$$
I_1(x,t)= \int_0^t \int_{\R^d}(t-r)^{\alpha-1}\Gamma(x,t,z,r) Y(z,r)  dz dr
$$
with $Y$ given by \eqref{eq:def_Y}. Note that we are not able to obtain boundedness of $Y$; to do so we would have to exchange the expectation and the supremum for an anticipating stochastic integral. Recall that according to  \ref{D3} we have $2pq/(p+q) > 2d+4$. Hence the constant $\kappa$ in Lemma \ref{lem:auxillary_alpha_process} can be chosen in such a way that $2<\kappa<2d+4$. Since $Y$ is not bounded, we will apply Lemma \ref{lem:reg_volume_potential_2}. Denote
\begin{eqnarray*}
Z(x,t) &=& \int_0^{t} \int_{\R^d}(t-r)^{\alpha-1}\Gamma(x,t,u,r) |Y(u,r)|^\delta  du dr,\\
\widehat Z(x,t) & = & \int_0^{t} \int_{\R^d}(t-r)^{\alpha-1} \nabla \Gamma(x,t,u,r) |Y(u,r)|^\delta  du dr.
\end{eqnarray*}
\begin{lemma}\label{lem:control_Z}
For any $\frac{1}{2} + \frac{(d+2)(p+q)}{4pq} < \alpha < \frac{q-1}{q}$, there exists $1 < \delta < \frac{2pq}{(p+q)}$ such that
$$\mE \left[  \sup_{x,t} \left( Z(x,t)\right)^{\frac{2pq}{(p+q)\delta}}  +  \sup_{x,t} \left( \widehat Z(x,t)\right)^{\frac{2pq}{(p+q)\delta}}  \right] \leq  C  \mE \int_0^{T} \int_{\R^d}  |Y(z,r)|^{\frac{2pq}{(p+q)}}  dz dr.$$
\end{lemma}
\begin{proof}
Choose $1 < \delta < \frac{2pq}{(p+q)}$ and $\theta =   \frac{2pq}{(p+q)\delta}> 1$. Then
\begin{eqnarray*}
&& \int_0^{t} \int_{\R^d}(t-r)^{\alpha-1}\Gamma(x,t,z,r) |Y(z,r)|^\delta  dz dr \\
&& \ \  \leq \int_0^{t} (t-r)^{\alpha-1} \left( \int_{\R^d} \Gamma(x,t,z,r)^{\frac{\theta}{\theta-1}} dz \right)^{\frac{\theta-1}{\theta}} \left( \int_{\R^d}  |Y(z,r)|^{\frac{2pq}{(p+q)}}dz \right)^{\frac{1}{\theta}} dr
\end{eqnarray*}
Due to the Aronson estimate the right-hand side here admits the following upper bound:
\begin{eqnarray*}
&&\int_0^{t} (t-r)^{\alpha-1} \left( \int_{\R^d} \Gamma(x,t,z,r)^{\frac{\theta}{\theta-1}} dz \right)^{\frac{\theta-1}{\theta}} \left( \int_{\R^d}  |Y(z,r)|^{\frac{2pq}{(p+q)}}dz \right)^{\frac{1}{\theta}} dr\\
&&\quad \leq C \int_0^{t} (t-r)^{\alpha-1-\frac{d}{2\theta}}  \left( \int_{\R^d}  |Y(z,r)|^{  \frac{2pq}{(p+q)}} dz \right)^{\frac{1}{\theta}} dr \\
&& \quad  \leq C \left(  \int_0^{t} (t-r)^{(\alpha-1-\frac{d}{2\theta})\frac{\theta}{\theta-1}} dr \right)^{\frac{\theta-1}{\theta}} \left( \int_0^{t} \int_{\R^d}  |Y(z,r)|^{ \frac{2pq}{(p+q)}} dz dr \right)^{\frac{1}{\theta}} \\
&& \quad  \leq C  \left( \int_0^{T} \int_{\R^d}  |Y(z,r)|^{ \frac{2pq}{(p+q)}} dz dr \right)^{\frac{1}{\theta}};
\end{eqnarray*}
here the latter inequality holds if
$\left(\alpha-1-\frac{d}{2\theta} \right)\frac{\theta}{\theta-1} > -1$, or equivalently $\alpha > \frac{d+2}{2\theta} = \frac{(d+2)(p+q)}{4pq} \delta$.\\
The computations similar to those in the proof of the previous lemma yield
$$
\int_0^{t} \int_{\R^d}(t-r)^{\alpha-1} \nabla \Gamma(x,t,z,r) |Y(z,r)|^\delta  dz dr
\leq C  \left( \int_0^{T} \int_{\R^d}  |Y(z,r)|^{ \frac{2pq}{(p+q)}} dz dr \right)^{\frac{1}{\theta}},
$$
if $\left(\alpha-1-\frac{1}{2}-\frac{d}{2\theta} \right)\frac{\theta}{\theta-1} > -1$, or equivalently $\alpha > \frac{1}{2} + \frac{d+2}{2\theta} = \frac{1}{2} + \frac{(d+2)(p+q)}{4pq} \delta$.

\end{proof}

From Lemmata \ref{lem:control_Z} and \ref{lem:reg_volume_potential_2} it follows that $I_1$ is a.s. continuous w.r.t. $(x,t)$ and differentiable w.r.t. $x$. Arguing as in the proof of the above lemma, we obtain
$$
\mE \left[\sup_{x,t}  \left( | I_1(x,t) |^{\frac{2pq}{p+q} } +  |\nabla I_1(x,t) |^{\frac{2pq}{p+q} } \right) \right] < +\infty.
$$
Furthermore, a careful examination of our proofs shows that there exists $\eta >0$ such that for any $h>0$
$$\mE \Big[ \sup_{x, 0\leq t\leq h}   | v(x,t) |^{\frac{2pq}{p+q}} \Big] \leq Ch^\eta.$$
This implies that a.s. for any $x \in \R^d$, $v(x,t)$ tends to zero as $t$ goes to zero.

To complete the proof of Theorem \ref{thm:reg_mild_sol} consider a function $\phi \in C^\infty_0(\R^d)$ and
$$J(t) =  \int_{\R^d} v(x,t) \phi(x) dx + \int_0^t \int_{\R^d} \rma(x,u) \nabla v(x,u) \nabla \phi(x) dx du .$$
By the previous Lemmata, $J(t)$ is well defined on $[0,T]$ with
\begin{eqnarray*}
J(t)& = &\int_{\R^d} \left(\int_0^t   \int_{\R^d} \Gamma(x,t,y,s) G(y,s) dy dB_s \right) \phi(x) dx   \\
& + & \int_0^t \int_{\R^d} \rma(x,u)  \left( \int_0^u \int_{\R^d} \nabla \Gamma(x,u,y,s) G(y,s) dy dB_s \right) \nabla \phi(x) dx du.
\end{eqnarray*}
By the Fubini theorem
\begin{eqnarray*}
J(t)& = &\int_0^t  \int_{\R^d} \left(\int_{\R^d} \Gamma(x,t,y,s)  \phi(x) dx \right) G(y,s) dy dB_s \\
& + & \int_0^t \int_{\R^d} \left( \int_s^t \int_{\R^d} \nabla \Gamma(x,u,y,s) a(x,u) \nabla \phi(x) dx du \right) G(y,s) dy dB_s \\
& = & \int_0^t  \int_{\R^d}   \phi(y) G(y,s) dy dB_s,
\end{eqnarray*}
since $\Gamma$ is the fundamental solution of \eqref{eq:PDE_fund_sol}.

\section*{Appendix}

\setcounter{lemmaA}{0}
\renewcommand{\thelemmaA}{A.\arabic{lemmaA}}

Recall that
$$V(x,t) = \int_0^t \int_{\R^d} \Gamma(x,t,y,s) f(y,s) dy ds$$
is the volume potential of $f$ (see \cite[Section I.3]{frie:64}). Here we give some results concerning the regularity of $V$. The first lemma is closely related to Lemma I.3.1 and Theorem I.3.3 of \cite{frie:64} and Theorem 1 of \cite{ilyi:kala:oley:02}.
\begin{lemmaA} \label{lem:reg_volume_potential}
Assume that $f$ is a bounded measurable function. Then $V$ is continuous w.r.t. $(x,t) \in \R^d \times (0,+\infty)$ and has first continuous derivatives w.r.t. $x$. Moreover, for any $t > 0$ and $x\in \R^d$,
$$\frac{\partial V}{\partial x_i} (x,t) = \int_0^t \int_{\R^d} \frac{\partial }{\partial x_i} \Gamma(x,t,y,s) f(y,s) dy ds.$$
\end{lemmaA}
\begin{proof}
Fix some $x \in \R^d$ and $t > 0$ and consider
$$J(x,t,s) = \int_{\R^d} \Gamma(x,t,y,s) f(y,s) dy.$$
This function is continuous with respect to all its arguments $x \in \R^d$ and $0 \leq s < t$. Moreover,  by \eqref{eq:aronson_estim}
$$
|J(x,t,s)| \leq  \int_{\R^d} g_{\varrho,\varpi}(x-y,t-s) |f(y,s)|  dy \leq C.
$$
Since the function $\int_0^{t-\varepsilon} J(x,t,s) ds$ is continuous for any sufficiently small $\varepsilon>0$,  this implies the required
continuity of  $V$. For the derivatives, let us consider
$$J(x,t,s) = \int_{\R^d} \Gamma(x,t,y,s) g(y,s) dy.$$
For any $s<t$, it holds
$$\frac{\partial J}{\partial x_i} (x,t,s) = \int_{\R^d} \frac{\partial }{\partial x_i} \Gamma(x,t,y,s) g(y,s) dy.$$
Now using \eqref{eq:aronson_estim_der}, we have
$$ \left| \frac{\partial J}{\partial x_i} (x,t,s) \right| \leq (t-s)^{-1/2} \int_{\R^d} g_{\varrho,\varpi}(x-y,t-s) |g(y,s)| dy \leq C(t-s)^{-1/2}.$$
Therefore the integral
$$\int_0^t \frac{\partial J}{\partial x_i} (x,t,s) ds$$
converges uniformly with respect to $x$ and $t>0$. It follows that for $t > 0$ and any $x$, the derivatives
$$\frac{\partial V}{\partial x_i} (x,t) = \int_0^t \frac{\partial J}{\partial x_i} (x,t,s) ds$$
exist and are continuous.
\end{proof}

Let us give another version of these results.
\begin{lemmaA}\label{lem:reg_volume_potential_2}
Let $g$ be a measurable function such that for some $q > 1$ there exists a constant $K \geq 0$ such that for any $(x,t) \in \R^d \times (0,+\infty)$
$$\int_0^t \int_{\R^d} \left[ \Gamma(x,t,y,s) + \left| \nabla \Gamma(x,t,y,s) \right| \right]  |g(y,s)|^q dy ds \leq K.$$
Then $V$ is continuous w.r.t. $(x,t) \in \R^d \times (0,+\infty)$ and has first continuous derivatives w.r.t. $x$. Moreover, for any $t > 0$ and $x\in \R^d$,
$$\frac{\partial V}{\partial x_i} (x,t) = \int_0^t \int_{\R^d} \frac{\partial }{\partial x_i} \Gamma(x,t,y,s) g(y,s) dy ds.$$
\end{lemmaA}
\begin{proof}
By the H\"older inequality
$$|J(x,t,s)| \leq \left(  \int_{\R^d} \Gamma(x,t,y,s) dy \right)^{\frac{q-1}{q}} \left(  \int_{\R^d} \Gamma(x,t,y,s) |g(y,s)|^q dy \right)^{\frac{1}{q}}.$$
This implies the uniform convergence of the integral $\int_0^t J(x,t,s) ds$ w.r.t. $x$ and $t> 0$. Therefore, $V$ is continuous for $t>0$.  For the derivative, the same arguments give:
\begin{eqnarray*}
&& \left| \frac{\partial J}{\partial x_i} (x,t,s) \right| \\
&& \quad \leq \left(  \int_{\R^d} \left|\frac{\partial }{\partial x_i}\Gamma(x,t,y,s) \right|dy \right)^{\frac{q-1}{q}} \left(  \int_{\R^d} \left|\frac{\partial }{\partial x_i}\Gamma(x,t,y,s) \right| |g(y,s)|^q dy \right)^{\frac{1}{q}} \\
&& \quad \leq C(t-s)^{-(q-1)/(2q)} .
\end{eqnarray*}
The rest of the proof is exactly the same as in the previous lemma.
\end{proof}

\bigskip\noindent
{\bf Acknowledgements.}  The work of the second author was partially supported by Russian Science Foundation, project number 14-50-00150.

\bibliographystyle{plain}
\bibliography{recherchebib}

\def\cprime{$'$} \def\cprime{$'$} \def\cprime{$'$} \def\cprime{$'$}
\begin{thebibliography}{10}

\bibitem{alos:leon:nual:99}
E.~Al{\`o}s, J.~A. Le{\'o}n, and D.~Nualart.
\newblock Stochastic heat equation with random coefficients.
\newblock {\em Probab. Theory Related Fields}, 115(1):41--94, 1999.

\bibitem{aron:68}
D.~G. Aronson.
\newblock Non-negative solutions of linear parabolic equations.
\newblock {\em Ann. Scuola Norm. Sup. Pisa (3)}, 22:607--694, 1968.

\bibitem{barlow2003kernels}
Martin~T. Barlow.
\newblock Heat kernels and sets with fractal structure.
\newblock In {\em Heat kernels and analysis on manifolds, graphs, and metric
  spaces ({P}aris, 2002)}, volume 338 of {\em Contemp. Math.}, pages 11--40.
  Amer. Math. Soc., Providence, RI, 2003.

\bibitem{baud:14}
F.~Baudoin.
\newblock {\em Diffusion processes and stochastic calculus}.
\newblock EMS Textbooks in Mathematics. European Mathematical Society (EMS),
  Z\"urich, 2014.

\bibitem{chen_at_al:2017}
Z.-Q. Chen, E.~Hu, L.~Xie, and X.~Zhang.
\newblock Heat kernels for non-symmetric diffusion operators with jumps.
\newblock {\em J. Differential Equations}, 263(10):6576--6634, 2017.

\bibitem{chen:hu:xie:17}
Zhen-Qing Chen, Eryan Hu, Longjie Xie, and Xicheng Zhang.
\newblock Heat kernels for non-symmetric diffusion operators with jumps.
\newblock {\em J. Differential Equations}, 263(10):6576--6634, 2017.

\bibitem{chen:kuma:wang:19}
Zhen-Qing Chen, Takashi Kumagai, and Jian Wang.
\newblock Elliptic {H}arnack inequalities for symmetric non-local {D}irichlet
  forms.
\newblock {\em J. Math. Pures Appl. (9)}, 125:1--42, 2019.

\bibitem{chen:zhang:16}
Zhen-Qing Chen and Xicheng Zhang.
\newblock Heat kernels and analyticity of non-symmetric jump diffusion
  semigroups.
\newblock {\em Probab. Theory Related Fields}, 165(1-2):267--312, 2016.

\bibitem{chow:jiang:94}
Pao~Liu Chow and Jing-Lin Jiang.
\newblock Stochastic partial differential equations in {H}\"{o}lder spaces.
\newblock {\em Probab. Theory Related Fields}, 99(1):1--27, 1994.

\bibitem{coul:jiang:kosk:19}
Thierry Coulhon, Renjin Jiang, Pekka Koskela, and Adam Sikora.
\newblock Gradient estimates for heat kernels and harmonic functions.
\newblock {\em Journal of Functional Analysis}, 278(8):1--67, 11 2019.

\bibitem{dapr:zabc:14}
G.~Da~Prato and J.~Zabczyk.
\newblock {\em Stochastic equations in infinite dimensions}, volume 152 of {\em
  Encyclopedia of Mathematics and its Applications}.
\newblock Cambridge University Press, Cambridge, second edition, 2014.

\bibitem{dala:khos:09}
R.~Dalang, D.~Khoshnevisan, C.~Mueller, D.~Nualart, and Y.~Xiao.
\newblock {\em A minicourse on stochastic partial differential equations},
  volume 1962 of {\em Lecture Notes in Mathematics}.
\newblock Springer-Verlag, Berlin, 2009.
\newblock Held at the University of Utah, Salt Lake City, UT, May 8--19, 2006,
  Edited by Khoshnevisan and Firas Rassoul-Agha.

\bibitem{Daners_2000}
Daniel Daners.
\newblock Heat kernel estimates for operators with boundary conditions.
\newblock {\em Mathematische Nachrichten}, 217(1):13--41, 2000.

\bibitem{DAVIES198816}
E.B Davies.
\newblock Gaussian upper bounds for the heat kernels of some second-order
  operators on riemannian manifolds.
\newblock {\em Journal of Functional Analysis}, 80(1):16 -- 32, 1988.

\bibitem{deni:mato:stoi:05}
Laurent Denis, Anis Matoussi, and Lucretiu Stoica.
\newblock {$L^p$} estimates for the uniform norm of solutions of quasilinear
  {SPDE}'s.
\newblock {\em Probab. Theory Related Fields}, 133(4):437--463, 2005.

\bibitem{eide:zhit:98}
S.~D. Eidelman and N.~V. Zhitarashu.
\newblock {\em Parabolic boundary value problems}, volume 101 of {\em Operator
  Theory: Advances and Applications}.
\newblock Birkh\"auser Verlag, Basel, 1998.
\newblock Translated from the Russian original by Gennady Pasechnik and Andrei
  Iacob.

\bibitem{frie:64}
A.~Friedman.
\newblock {\em Partial differential equations of parabolic type}.
\newblock Prentice-Hall, Inc., Englewood Cliffs, N.J., 1964.

\bibitem{GRIGORYAN1995363}
A.~Grigoryan.
\newblock Upper bounds of derivatives of the heat kernel on an arbitrary
  complete manifold.
\newblock {\em Journal of Functional Analysis}, 127(2):363 -- 389, 1995.

\bibitem{grig:09}
A.~Grigor'yan.
\newblock {\em {Heat kernel and analysis on manifolds. AMS/IP Studies in
  advanced mathematics. V. {\bf 47}}}.
\newblock {AMS}, 2009.

\bibitem{Grigoryan1994}
Alexander Grigor'yan.
\newblock Heat kernel upper bounds on a complete non-compact manifold.
\newblock {\em Revista Matem{\'a}tica Iberoamericana}, 10(2):395--452, 1994.

\bibitem{grig:12}
Alexander Grigor'yan and Andras Telcs.
\newblock Two-sided estimates of heat kernels on metric measure spaces.
\newblock {\em Ann. Probab.}, 40(3):1212--1284, 2012.

\bibitem{Guenther2002}
Christine~M. Guenther.
\newblock The fundamental solution on manifolds with time-dependent metrics.
\newblock {\em The Journal of Geometric Analysis}, 12(3):425 -- 436, 2002.

\bibitem{hu:li:10}
Jun-Qi Hu and Hong-Quan Li.
\newblock Gradient estimates for the heat semigroup on {H}-type groups.
\newblock {\em Potential Anal.}, 33(4):355--386, 2010.

\bibitem{ilyi:kala:oley:02}
A.~M. Il\cprime~in, A.~S. Kalashnikov, and O.~A. Ole{\u\i}nik.
\newblock Second-order linear equations of parabolic type.
\newblock {\em Tr. Semin. im. I. G. Petrovskogo}, (21):9--193, 341, 2001.

\bibitem{JIAYU1991293}
Li~Jiayu.
\newblock Gradient estimate for the heat kernel of a complete riemannian
  manifold and its applications.
\newblock {\em Journal of Functional Analysis}, 97(2):293 -- 310, 1991.

\bibitem{kryl:96}
N.~V. Krylov.
\newblock On {$L_p$}-theory of stochastic partial differential equations in the
  whole space.
\newblock {\em SIAM J. Math. Anal.}, 27(2):313--340, 1996.

\bibitem{kryl:99}
N.~V. Krylov.
\newblock An analytic approach to {SPDE}s.
\newblock In {\em Stochastic partial differential equations: six perspectives},
  volume~64 of {\em Math. Surveys Monogr.}, pages 185--242. Amer. Math. Soc.,
  Providence, RI, 1999.

\bibitem{kryl:08}
N.~V. Krylov.
\newblock {\em Lectures on elliptic and parabolic equations in {S}obolev
  spaces}, volume~96 of {\em Graduate Studies in Mathematics}.
\newblock American Mathematical Society, Providence, RI, 2008.

\bibitem{kryl:14}
N.~V. Krylov.
\newblock H\"{o}rmander's theorem for parabolic equations with coefficients
  measurable in the time variable.
\newblock {\em SIAM J. Math. Anal.}, 46(1):854--870, 2014.

\bibitem{kryl:rozo:79}
N.~V. Krylov and B.~L. Rozovski\u\i.
\newblock It\^o equations in {B}anach spaces and strongly parabolic stochastic
  partial differential equations.
\newblock {\em Dokl. Akad. Nauk SSSR}, 249(2):285--289, 1979.

\bibitem{lady:solo:ural:67}
O.~A. Lady{\v{z}}enskaja, V.~A. Solonnikov, and N.~N. Ural{\cprime}ceva.
\newblock {\em Linear and quasilinear equations of parabolic type}.
\newblock Translated from the Russian by S. Smith. Translations of Mathematical
  Monographs, Vol. 23. American Mathematical Society, Providence, R.I., 1967.

\bibitem{MR3809455}
Janna Lierl.
\newblock Parabolic {H}arnack inequality on fractal-type metric measure
  {D}irichlet spaces.
\newblock {\em Rev. Mat. Iberoam.}, 34(2):687--738, 2018.

\bibitem{lierl:18}
Janna Lierl.
\newblock Parabolic {H}arnack inequality on fractal-type metric measure
  {D}irichlet spaces.
\newblock {\em Rev. Mat. Iberoam.}, 34(2):687--738, 2018.

\bibitem{LIERL20144189}
Janna Lierl and Laurent Saloff-Coste.
\newblock The dirichlet heat kernel in inner uniform domains: Local results,
  compact domains and non-symmetric forms.
\newblock {\em Journal of Functional Analysis}, 266(7):4189 -- 4235, 2014.

\bibitem{miku:prag:13}
R.~Mikulevicius and H.~Pragarauskas.
\newblock On {$L_p$}-theory for stochastic parabolic integro-differential
  equations.
\newblock {\em Stoch. Partial Differ. Equ. Anal. Comput.}, 1(2):282--324, 2013.

\bibitem{miku:prag:14}
R.~Mikulevicius and H.~Pragarauskas.
\newblock On the {C}auchy problem for integro-differential operators in
  {H}\"{o}lder classes and the uniqueness of the martingale problem.
\newblock {\em Potential Anal.}, 40(4):539--563, 2014.

\bibitem{miku:rozo:01}
R.~Mikulevicius and B.~Rozovskii.
\newblock A note on {K}rylov's {$L_p$}-theory for systems of {SPDE}s.
\newblock {\em Electron. J. Probab.}, 6:no. 12, 35, 2001.

\bibitem{nual:06}
D.~Nualart.
\newblock {\em The {M}alliavin calculus and related topics}.
\newblock Probability and its Applications (New York). Springer-Verlag, Berlin,
  second edition, 2006.

\bibitem{nual:pard:88}
D.~Nualart and {\'E}.~Pardoux.
\newblock Stochastic calculus with anticipating integrands.
\newblock {\em Probab. Theory Related Fields}, 78(4):535--581, 1988.

\bibitem{pard:93}
\'E. Pardoux.
\newblock Stochastic partial differential equations, a review.
\newblock {\em Bull. Sci. Math.}, 117(1):29--47, 1993.

\bibitem{pasc:pesc:18}
Andrea Pascucci and Antonello Pesce.
\newblock The parametrix method for parabolic spdes.
\newblock {\em arXiv preprint arXiv:1803.06543}, 2018.

\bibitem{porp:eide:84}
F.~O. Porper and S.~D. \`E\u\i~del\cprime man.
\newblock Two-sided estimates of the fundamental solutions of second-order
  parabolic equations and some applications of them.
\newblock {\em Uspekhi Mat. Nauk}, 39(3(237)):107--156, 1984.

\bibitem{saloff-coste2010}
Laurent Saloff-Coste.
\newblock The heat kernel and its estimates.
\newblock In {\em Probabilistic Approach to Geometry}, pages 405--436, Tokyo,
  Japan, 2010. Mathematical Society of Japan.

\bibitem{sowe:98}
Richard~B. Sowers.
\newblock Short-time geometry of random heat kernels.
\newblock {\em Mem. Amer. Math. Soc.}, 132(629):viii+130, 1998.

\bibitem{Stroock1998UpperBO}
Daniel~W. Stroock and James Turetsky.
\newblock Upper bounds on derivatives of the logarithm of the heat kernel.
\newblock {\em Communications in Analysis and Geometry}, 6(4):669 -- 685, 1998.

\bibitem{MR1218884}
N.~Th. Varopoulos, L.~Saloff-Coste, and T.~Coulhon.
\newblock {\em Analysis and geometry on groups}, volume 100 of {\em Cambridge
  Tracts in Mathematics}.
\newblock Cambridge University Press, Cambridge, 1992.

\bibitem{wals:86}
J.~B. Walsh.
\newblock An introduction to stochastic partial differential equations.
\newblock In {\em \'Ecole d'\'et\'e de probabilit\'es de {S}aint-{F}lour,
  {XIV}---1984}, volume 1180 of {\em Lecture Notes in Math.}, pages 265--439.
  Springer, Berlin, 1986.

\end{thebibliography}

\end{document}